\documentclass[twoside,10pt]{amsart}

\usepackage[english]{babel}
\usepackage{amsmath,amsfonts,amssymb,amsthm}
\usepackage{mathrsfs}
\usepackage{amscd}
\usepackage{graphics,color,graphicx,pstricks}

\theoremstyle{plain}
\newtheorem{thm}{Theorem}[section]
\newtheorem{lem}[thm]{Lemma}
\newtheorem{prop}[thm]{Proposition}
\newtheorem{cor}[thm]{Corollary}

\theoremstyle{plain}
\newtheorem{thmIntro}{Theorem}

\theoremstyle{definition}
\newtheorem{defn}{Definition}[section]

\theoremstyle{remark}
\newtheorem{rem}{Remark}[section]

\newtheorem{nota}{Notation}

\newcommand{\Z}{\mathbb{Z}}
\newcommand{\Q}{\mathbb{Q}}
\newcommand{\R}{\mathbb{R}}
\newcommand{\C}{\mathbb{C}}

\newcommand{\got}[1]{\mathfrak{#1}}
\newcommand{\Ad}{\mathop{\mathrm{Ad}}}
\newcommand{\ad}{\mathop{\mathrm{ad}}}
\newcommand{\tr}{\mathop{\mathrm{Tr}}}
\newcommand{\diag}{\mathrm{diag}}

\newcommand{\Orb}{\mathcal{O}}
\newcommand{\Mat}{\mathcal{M}at}

\newcommand{\Chol}{\mathcal{C}_{\mathrm{hol}}}
\newcommand{\Pic}{\mathrm{Pic}}
\newcommand{\PicG}{\mathrm{Pic}^{\KC}}
\newcommand{\PicGQ}{\mathrm{Pic}^{\KC}_{\Q}}
\newcommand{\Xss}{X^{\mathrm{ss}}}
\newcommand{\XssL}{X^{\mathrm{ss}}(\mathcal{L})}

\newcommand{\id}{\mathrm{id}}
\newcommand{\KC}{K_{\C}}
\newcommand{\TC}{T_{\C}}
\newcommand{\hKC}{\hat{K}_{\C}}
\newcommand{\hTC}{\hat{T}_{\C}}
\newcommand{\tKC}{\tilde{K}_{\C}}
\newcommand{\tTC}{\tilde{T}_{\C}}
\newcommand{\pt}{[\mathrm{pt}]}
\newcommand{\wh}{\hat}
\newcommand{\wt}{\tilde}
\newcommand{\halpha}{\hat{\alpha}}

\newcommand{\bmin}{\beta_{\mathrm{min}}}
\newcommand{\XEC}{X_{E\oplus\C}}

\newcommand{\CR}{\mathrm{Cone}_{\mathbb{R}}}
\newcommand{\WT}{\mathcal{W}_{\TC}}

\newcommand{\coh}[1]{\mathrm{H}^*(#1,\Z)}
\newcommand{\homol}[1]{\mathrm{H}_*(#1,\Z)}

\begin{document}

\title{Kirwan polyhedron of holomorphic coadjoint orbits}

\author{Guillaume~DELTOUR}
\address{Universit\'e Montpellier 2, CC051\\
Place Eug\`ene Bataillon\\
34095 Montpellier Cedex}
\email{gdeltour@math.univ-montp2.fr
}
\date{\today}

\begin{abstract}
Let $G$ be a simple, noncompact, connected, real Lie group with finite center, and $K$ a maximal compact subgroup of $G$. 
We assume that $G/K$ is Hermitian. Using GIT methods derived from the generalized eigenvalue problem, we compute a set of affine equations describing the moment polyhedron of the projection $G\cdot\Lambda\subset\got{g}^*\rightarrow\got{k}^*$ for holomorphic coadjoint orbits of $G$.
\end{abstract}

\maketitle


\section*{Introduction}

In the last 15 years, important breakthroughs have been made in the study of compact orbit projection. This was initiated by the resolution of Horn's famous conjecture, by Klyachko \cite{Klyachko98} and Knutson-Tao-Woodward \cite{Knutson-Tao-Woodward}.

However, the noncompact case is still misunderstood. Hilgert-Neeb-Plank \cite{Hilgert-Neeb-Plank} and Eshmatov-Foth \cite{Eshmatov-Foth} have been able to describe the moment polyhedra of special noncompact orbit projections, as the sum of a convex polytope and a convex cone generated by roots of the Lie algebra. Unfortunately, the generic formulas of the equations of these noncompact orbit projections are unknown.

In \cite{duflo84}, Duflo-Heckman-Vergne computed the pushforward of the Liouville measure by the orbit projection of any regular elliptic orbit $\mathcal{O}$ of certain reductive Lie groups, as an alternate sum of measures supported by cones. The Kirwan polyhedron $\Delta_K(\mathcal{O})$ is exactly the support of this measure. However, similarly to the Kostant formula, the Duflo-Heckman-Vergne formula does not allow to explicitly describe its support in general.

In this paper, we study the equations of the moment polyhedron associated to the orbit projection of another type of noncompact coadjoint orbits.

\medskip
Let $G$ be a connected real Lie group, and $K$ a compact connected Lie subgroup of $G$. Let $\got{g}$ and $\got{k}$ denote the corresponding Lie algebras. Any coadjoint orbit $\Orb\in\got{g}^*$ of $G$, endowed with its Kirillov-Kostant-Souriau symplectic form, is naturally a Hamiltonian $K$-manifold. The standard moment map is the \emph{orbit projection} $\Phi_K:\Orb\rightarrow\got{k}^*$, which is the composition of the injection $\Orb\subseteq\got{g}^*$ with the linear projection $\got{g}^*\rightarrow\got{k}^*$. When $\Phi_K$ is a proper map, a noncompact version of Kirwan's Hamiltonian convexity \cite{lerman_et_al,sjamaar_convexity} asserts that the image $\Phi_K(\Orb)$ intersects some Weyl chamber $\got{t}_+^*$ of $K$ into a set, denoted by $\Delta_K(\Orb)$, which is convex locally polyhedral. The set $\Delta_K(\Orb)$ is called the \emph{Kirwan (\emph{or} moment) polyhedron} of the projection of the orbit $\Orb$.

Now assume that $G$ is also compact, and $\Lambda$ is a weight of $G$. Then, the coadjoint orbit $\Orb_{\Lambda}=G\cdot\Lambda$ is a prequantizable K\"ahler $K$-manifold, and its Kirwan polytope may be described in terms of irreducible representations of $K$ and $G$. Indeed, $\Delta_K(\Orb_{\Lambda})$ is the closure in $\got{t}^*_+$ of the following rational polytope
\[
\{\mu \text{ dominant rational weight of $K$}\,|\,\exists N\geq 1 \text{ integer s.t. } V_{N\mu}^K\subseteq V_{N\Lambda}^G\}.
\]
Here, $V_{\nu}^K$ (resp. $V_{\nu}^G$) denotes the irreducible representation of $K$ (resp. $G$) of highest weight $\nu$. It happens that this rational polytope is (roughly) an affine section of a bigger polyhedral convex cone,
\[
\{(\mu,\nu) \text{ dominant rational weight of $K\times G$}\,|\,\exists N\geq 1 \text{ s.t. } (V_{N\mu}^{K*}\otimes V_{N\nu}^{G*})^K\neq 0\},
\]
called the \emph{semiample cone} of the complete flag variety of $K\times G$, see \cite{Dolgachev-Hu, ressayre10}. Ressayre's recent results \cite{ressayre10} gives a (minimal) set of equations of such semiample cone, using the notion of well covering pairs on the complete flag variety of $K\times G$. The equations of $\Delta_K(\Orb_{\Lambda})$ follows from the one of the semiample cone.

A (nonminimal) set of equations of $\Delta_K(\Orb_{\Lambda})$ has also been given by Berenstein-Sjamaar in \cite{BS00}, making use of Hilbert-Mumford criterion directly on the complete flag of $G$, acted on by left multiplication of $K$.

\medskip
When $G$ is not compact, we cannot apply such method, at least not directly on the complete flag of $K\times G$. Nevertheless, in some special cases of noncompact groups $G$ and coadjoint orbits of $G$, it is possible to obtain the formulas of the equations of $\Delta_K(\Orb)$ by applying the previous techniques on a good compactification of the coadjoint orbit $\Orb$. The best way to find coadjoint orbits on which we can apply such method, is to consider orbits satisfying similar hypotheses to the compact setting. That is, we want $\Orb$ to be a prequantizable K\"ahler manifold.

\medskip
Now, assume that $G$ is a semisimple, noncompact, connected, real Lie group with finite center, and let $K$ be the maximal compact subgroup of $G$ arising from a Cartan decomposition $\got{g} = \got{k}\oplus\got{p}$ on Lie algebra level. We also assume that $G/K$ is a Hermitian symmetric space.

Among the integral elliptic coadjoint orbits of $G$, some of them are naturally prequantizable K\"ahler $K$-manifolds. These orbits are called the \emph{holomorphic coadjoint orbits} of $G$. They are the strongly elliptic coadjoint orbits closely related to the holomorphic discrete series of Harish-Chandra. These orbits intersect the Weyl chamber $\got{t}^*_+$ of $K$ into a subchamber called the \emph{holomorphic chamber} $\Chol$. If $\Lambda$ is an element of $\Chol$, then $\Orb_{\Lambda}=G\cdot\Lambda$ admits a simple symplectic model, given by the product of symplectic manifolds $(K\cdot\Lambda\times\got{p}, \Omega_{K\cdot\Lambda}\oplus\Omega_{\got{p}})$ \cite{deltour,mcduff}. See section \ref{section:projection_of_holomorphic_coadjoint_orbits} for more details about these facts.

We thus consider a new Kirwan polyhedron $\Delta_K(K\cdot\Lambda\times\got{p})$. It appears that a good compactification of $\Orb_{\Lambda}$ is the projective variety $K\cdot\Lambda\times\mathbb{P}(\got{p}\oplus\C)$.

Following this idea, we are able to determine the equations of $\Delta_K(\Orb_{\Lambda})$, by computing the ones of the semiample cone of $K/T\times K/T\times\mathbb{P}(\got{p}\oplus\C)$, where $T$ is a maximal torus in $K$.

Let $\got{t}$ denote the Lie algebra of $T$, $W$ the Weyl group of $T$ in $K$. Assume also that $\got{t}_+^*$ is a Weyl chamber of $K$ in $\got{t}^*$.

\begin{thmIntro}
\label{thmIntro:thmA}
Let $\Lambda$ be in $\Chol$. There exists a finite subset $\mathscr{P}_0$ of $W\times W\times \got{t}$ such that, for all $\Lambda\in\Chol$, an element $\xi\in\got{t}^*_+$ is in $\Delta_K(\mathcal{O}_{\Lambda})$ if and only if $\xi$ satisfies the equation $\langle w\lambda,\xi\rangle \leq \langle w_0w'\lambda,\Lambda\rangle$ 
for all $(w,w',\lambda)\in\mathscr{P}_0$.
\end{thmIntro}

The set $\mathscr{P}_0$ is defined by Ressayre's well covering pairs on the flag variety $K/T\times K/T\times\mathbb{P}(\got{p}\oplus\C)$. We have a criterion which allows to compute these well covering pairs. This criterion is stated in Theorem \ref{thmIntro:thmB}.

For any $\lambda\in\got{t}$, let $W_{\lambda}$ be the stabilizer of $\lambda$. We define $\{\sigma_{w}^B; w\in W\}$ to be the dual basis in $\coh{\KC/B}$ of the Schubert basis consisting of the fundamental classes of the Schubert varieties. Here, $\KC$ (resp. $\TC$) is the complexification of $K$ (resp. $T$) and $B$ is a Borel subgroup of $\KC$. We denote by $w_0$ (resp. $w_{\lambda}$) the longest element in $W$ (resp. $W_{\lambda}$). For any weight $\mu$ of $K$, let $\Theta(\mu)=c_1(\mathcal{L}_{\mu})$ be the first Chern class of the line bundle $\mathcal{L}_{\mu}$ on $K/T$ with weight $\mu$.

\begin{thmIntro}
\label{thmIntro:thmB}
Let $(w,w',\lambda)\in W\times W\times\got{t}$. The triple $(w,w',\lambda)$ is in $\mathscr{P}_0$ if and only if
\begin{enumerate}
\item $\lambda$ is a dominant indivisible one parameter subgroup of $\TC$,
\item $\C\lambda = \cap_{\beta\in I}\ker\beta$, for some subset $I$ of the set $\got{R}_n^+$ of noncompact positive roots,
\item $\sigma_{w_0w}^B\,.\,\sigma_{w_0w'}^B\,.\,\prod_{\beta\in\got{R}_n^+, \langle\lambda,\beta\rangle>0}\Theta(\beta) = \sigma_{w_0w_{\lambda}}^B$,
\item $\langle w\lambda+w'\lambda,\rho\rangle=\sum_{\beta\in\got{R}_n^+}\langle\lambda,\beta\rangle$.
\end{enumerate}
\end{thmIntro}

%

In section 1, we reduce our problem to finding equations of a rational polyhedral convex cone $\Pi_{\Q}(\got{p})$ defined by tensor products of irreducible representations of $K$. Section 2 is a collection of prerequisite GIT facts. The relation between $\Pi_{\Q}(\got{p})$ and the semiample cone of $X_{\got{p}\oplus\C}=K/T\times K/T\times\mathbb{P}(\got{p}\oplus\C)$ is then shown in section 3. We give general equations for $C_{\Q}(X_{\got{p}\oplus\C})^+$, but these equations may be redundant, thus we give two conditions on $\lambda$ to exclude some of them in section 4, proving $\mathscr{P}_0$ is finite. The main result of this paper, Theorem \ref{thmIntro:thmA}, is proved in section 5. In section 6, we complete the proof of the main criterion stated in Theorem \ref{thmIntro:thmB}. Then, we compute the set of equations for the examples $Sp(\R^{2n})$ and $SU(n,1)$, for all $n\geq 2$, and $SU(2,2)$. The Appendix collects some technical results about combinatorics of the Weyl group of $GL_r(\C)$.

\section{Projection of holomorphic coadjoint orbits}
\label{section:projection_of_holomorphic_coadjoint_orbits}

\subsection{The holomorphic chamber $\Chol$}
\label{subsection:holomorphic_chamber}

Let $G$ be a noncompact, connected, real, semisimple Lie group with finite center, and $K$ a maximal compact subgroup of $G$. Let $\got{g}$ and $\got{k}$ denote their Lie algebras. The Lie subalgebra $\got{k}$ of $\got{g}$ arises from a Cartan decomposition on the Lie algebra level, $\got{g}=\got{k}\oplus\got{p}$. Moreover, $K$ is connected (see for instance \cite[Theorem 6.31]{knapp}). The vector subspace $\got{p}$, called the \emph{noncompact part} of the Cartan decomposition of $\got{g}$, is stable by the adjoint action of $K$ on $\got{g}$.

We recall that the symmetric space $G/K$ is Hermitian if it admits a complex-manifold structure such that $G$ acts by holomorphic transformations. The following assertions are equivalent:
\begin{enumerate}
\item $G/K$ is Hermitian,
\item there exists $z_0$ in the center of $\got{k}$ such that $\ad(z_0)|_{\got{p}}^2 = -\id_{\got{p}}$.
\end{enumerate}
A proof of this equivalence is given by Theorems 7.117 and 7.119 in \cite{knapp}.

We have a complete classification of simple groups $G$ (satisfying the above hypotheses) with $G/K$ Hermitian, up to isomorphism: $Sp(\R^{2n}), n\geq 1$, $SO^*(2n), n\geq 3$, $SU(p,q), p\geq q\geq 1$, and $SO_0(p,2), p\geq 1$ for the classical noncompact groups, and $E{\mathrm{III}}$ and $E{\mathrm{VII}}$ for the exceptional cases. One can find this classification in \cite{knapp}.

Now assume $G/K$ is Hermitian, and let $z_0$ be an element of the center of $\got{k}$ such that $\ad(z_0)|_{\got{p}}^2 = -\id_{\got{p}}$. It means that $\ad(z_0)|_{\got{p}}$ defines a $K$-invariant $\C$-vector space structure on $\got{p}$. Denote by $\got{p}_{\C}$ the complexification of $\got{p}$, and similarly $\got{g}_{\C}$ and $\got{k}_{\C}$. The linear action of $K$ on $\got{p}$, defined by the adjoint action, induces a complex-linear action of $K$ on $\got{p}_{\C}$.

Denote by $\got{p}^{\pm,z_0}$ the eigenspace $\ker(\ad(z_0)|_{\got{p}_{\C}}\mp i)$ of $\ad(z_0)|_{\got{p}_{\C}}$ associated to the eigenvalue $\pm i$. Especially, $\ad(z_0)$ is multiplication by the complex number $\pm i$ on $\got{p}^{\pm,z_0}$. These two subspaces of $\got{p}_{\C}$ are $K$-stable.

Let $T$ be a maximal torus of $K$. We set the following convention: an element $\alpha\in\got{t}^*$ is a \emph{root} of $\got{g}$ (resp. $\got{k}$) if there exists $X\in\got{g}_{\C}$ (resp. $X\in\got{k}_{\C}$), $X\neq 0$, such that $[H,X] = i\alpha(H)X$ for all $H\in\got{t}$. The associated root space is
\[
\got{g}_{\alpha} := \{X\in\got{g}_{\C}\ |\ [H,X] = i\alpha(H)X, \forall H\in\got{t}\}.
\]
If $\alpha$ is a root of $\got{g}$, then either $\got{g}_{\alpha}\subseteq\got{k}_{\C}$ ($\alpha$ is said \emph{compact root}), or $\got{g}_{\alpha}\subseteq\got{p}_{\C}$ (\emph{noncompact root}). Note that the compact roots are the roots of the Lie algebra $\got{k}$. The set of compact (resp. noncompact) roots is denoted by $\got{R}_c$ (resp. $\got{R}_n$). Fix once and for all $\got{t}^*_+$ a Weyl chamber of $K$ in $\got{t}^*$, and let $\got{R}_c^+$ be the system of positive compact roots associated to this Weyl chamber. Notice that, since $z_0\in\got{t}$, for any noncompact root $\beta$, we have either $\got{g}_{\beta}\subseteq\got{p}^{+,z_0}$ (\emph{positive noncompact roots}) or $\got{g}_{\beta}\subseteq\got{p}^{-,z_0}$ (\emph{negative noncompact roots}). Denote by $\got{R}_n^{+,z_0}$ the set of positive noncompact roots of $\got{g}$. Then $\got{R}_c^+\cup\got{R}_n^{+,z_0}$ is a system of positive roots of $\got{g}$. Indeed, we can easily see that for all $\alpha\in\got{R}_c^+$, we have $\alpha(z_0)=0$, and, for all $\beta\in\got{R}_n^{+,z_0}$, $\beta(z_0) = 1$.

\begin{defn}
The \emph{holomorphic chamber} is the subchamber of $\got{t}^*_+$ defined by
\[
\Chol^{z_0}:=\{\xi\in\got{t}^*\ | \ (\beta,\xi)>0, \forall\beta\in\got{R}_n^{+,z_0}\},
\]
where $(\cdot,\cdot)$ is the inner product on $\got{t}^*$ induced by the Killing form on $\got{g}$. A \emph{holomorphic coadjoint orbit} is a coadjoint orbit $\Orb$ of $G$ which intersects $\Chol^{z_0}$ on a nonempty set.
\end{defn}

\begin{rem}
In the rest of the paper, we will assume that the element $z_0$ is fixed, so we will write $\Chol$, $\got{p}^{\pm}$ and $\got{R}_n^{\pm}$ instead of $\Chol^{z_0}$, $\got{p}^{\pm,z_0}$ and $\got{R}_n^{\pm,z_0}$.
\end{rem}

\medskip

Let $\Lambda\in\Chol$. The holomorphic coadjoint orbit $\Orb_{\Lambda}:=G\cdot\Lambda$ has a natural $G$-invariant K\"ahlerian structure:
\begin{enumerate}
\item a canonical $G$-invariant symplectic form $\Omega_{\Orb_{\Lambda}}$, called the Kirillov-Kostant-Souriau symplectic structure on $\Omega_{\Orb_{\Lambda}}$;
\item a $G$-invariant complex structure $J_{\Orb_{\Lambda}}$, which holomorphic tangent bundle $T^{1,0}(\Orb_{\Lambda})\rightarrow \Orb_{\Lambda}$ is equal, above $\Lambda$, to the $T$-submodule
\[
\sum_{\alpha\in\got{R}_c^+, (\alpha,\Lambda)\neq 0}\got{g}_{\alpha} + \underbrace{\sum_{\beta\in\got{R}_n^-}\got{g}_{\beta}}_{\got{p}^-}
\]
One can check that this complex structure is compatible with the symplectic form $\Omega_{\Orb_{\Lambda}}$.
\end{enumerate}
Besides, the stabilizer of $\Lambda$ is clearly compact, since $(\beta,\Lambda)\neq 0$ for all $\beta\in\got{R}_n$. More precisely, we have $G_{\Lambda} = K_{\Lambda}$.

\subsection{Kirwan polyhedron of the orbit projection}

Let $\Lambda$ be any element $\got{t}^*$, and let $\Orb_{\Lambda}$ be its coadjoint $G$-orbit. The induced action of the connected compact group $K$ on $\Orb_{\Lambda}$ is Hamiltonian, with moment map the orbit projection $\Phi_K:\Orb_{\Lambda}\rightarrow\got{k}^*$, which is the composition of the inclusion $\Orb_{\Lambda}\hookrightarrow\got{g}^*$ with the canonical linear projection $\got{g}^*\rightarrow\got{k}^*$. Since $\Orb_{\Lambda}$ is elliptic, its projection $\Phi_K$ is a proper map. In particular, a noncompact version of Kirwan's convexity theorem of Sjamaar \cite{sjamaar_convexity} (see also the version of Lerman {\it et al.} \cite{lerman_et_al}), asserts that the image of the orbit projection $\Phi_K$ intersects the Weyl chamber $\got{t}_+^*$ of $K$ into a convex locally polyhedral set $\Delta_K(\Orb_{\Lambda}) := \Phi_K(\Orb_{\Lambda})\cap\got{t}_+^*$, called the \emph{Kirwan polyhedron} of the orbit projection of $\Orb_{\Lambda}$.

Now assume that $\Lambda$ is in $\Chol$, that is, the coadjoint orbit $\Orb_{\Lambda}$ is holomorphic. Then, the stabilizer $G_{\Lambda}$ is compact, and the Cartan decomposition of $G$ induces a $K$-equivariant diffeomorphism $K\cdot\Lambda\times\got{p}\rightarrow \Orb_{\Lambda}, (k\Lambda, X)\mapsto e^Xk\Lambda$, where $K$ acts diagonally on $K\cdot\Lambda\times\got{p}$ (with the obvious actions on $K\cdot\Lambda$ and $\got{p}$).

The manifold $K\cdot\Lambda\times\got{p}$ carries a canonical symplectic structure $\Omega_{K\cdot\Lambda}\oplus\Omega_{\got{p}}$, arising from the direct product of the symplectic manifolds $(K\cdot\Lambda, \Omega_{K\cdot\Lambda})$ and $(\got{p},\Omega_{\got{p}})$, where $\Omega_{K\cdot\Lambda}$ is the Kirillov-Kostant-Souriau symplectic form on the compact coadjoint orbit $K\cdot\Lambda$, and $\Omega_{\got{p}}$ is the constant symplectic form defined on the vector space $\got{p}$ by
\[
\Omega_{\got{p}}(X,Y) = B_{\got{g}}(X,\ad(z_0)Y), \quad \text{for all } X,Y\in\got{p}.
\]
Here, $B_{\got{g}}$ denotes the Killing form on $\got{g}$.

\begin{thm}[Deltour \cite{deltour}, McDuff \cite{mcduff}]
\label{thm:symplecto_mcduff_generalized}
Let $\Lambda\in\Chol$. Then, there exists a $K$-equivariant diffeomorphism from $\Orb_{\Lambda}$ onto $K\cdot\Lambda\times\got{p}$ which takes the symplectic form $\Omega_{\Orb_{\Lambda}}$ on $\Orb_{\Lambda}$ to the symplectic form $\Omega_{K\cdot\Lambda}\oplus\Omega_{\got{p}}$.
\end{thm}

The symplectic manifold $(K\cdot\lambda\times\got{p},\Omega_{K\cdot\lambda\times\got{p}})$ also have a Hamiltonian $K$-manifold structure. It is given by the moment map defined for all $(\xi,v)\in K\cdot\Lambda\times\got{p}$ by
\[
\Phi_{K\cdot\Lambda\times\got{p}}(\xi,v) : X\in\got{k}\longmapsto\langle\xi,X\rangle + \frac{1}{2}\Omega_{\got{p}}(v,[X,v])\in\R
\]
This moment map is proper, so that we can define the associated Kirwan polyhedron $\Delta_K(K\cdot\Lambda\times\got{p}):=\Phi_{K\cdot\Lambda\times\got{p}}(K\cdot\Lambda\times\got{p})\cap\got{t}^*_+$. Theorem \ref{thm:symplecto_mcduff_generalized} has the following direct consequence, originally proved by Nasrin (for $\Lambda$ in the center of $\got{k}^*$) and Paradan in totally different ways \cite{nasrin,paradan}.

\begin{cor}[Nasrin \cite{nasrin}, Paradan \cite{paradan}]
\label{cor:equality_of_momentpolyhedra}
Let $\Lambda\in\Chol$. Then
\[
\Delta_K(G\cdot\Lambda) = \Delta_K(K\cdot\Lambda\times\got{p}).
\]
\end{cor}

This new description of the Kirwan polyhedron $\Delta_K(G\cdot\Lambda)$ will allow to describe its faces, using GIT methods on the second setting. This question will be dealt with in the rest of the paper.

More generally, we are going to determine the equations of the Kirwan polyhedron $\Delta_K(K\cdot\Lambda\times E)$, when $E$ is a complex representation of the compact group $K$, satisfying some specific assumptions.

\subsection{Kirwan polyhedron of the symplectic manifold $K\cdot\Lambda\times E$}
\label{subsection:Kirwanpolyhedron_KLambdaE}


Let $(E,h)$ be a Hermitian vector space, $U:=U(E,h)$ the unitary group associated to $(E,h)$, and $\got{u}$ the Lie algebra of $U$.
Let $\Omega_E$ be the imaginary part of $-h$. Then $\Omega_E$ is a constant symplectic structure on $E$, invariant by the action of $U$. The (real) symplectic vector space $(E,\Omega_E)$ is $U$-Hamiltonian, with moment map $\Phi_U:E\rightarrow \got{u}^*$ defined by $\langle\Phi_U(v),X\rangle = \frac{1}{2}\Omega_E(Xv,v)$, for all $X\in\got{u}^*$.

Let $\varphi:K\rightarrow U$ be a Lie group homomorphism, and $\Phi_E : E\rightarrow \got{k}^*$ the moment map induced by the composition of the map $\Phi_U$ with the transpose $^t(d\varphi):\got{u}^*\rightarrow\got{k}^*$.

\begin{prop}[\cite{paradan_fgq}, Lemma 5.2]
The following conditions are equivalent :
\begin{enumerate}
\item The map $\Phi_E$ is proper,
\item $\Phi_E^{-1}(0) = \{0\}$.
\end{enumerate}
\end{prop}

In \cite{paradan_fgq}, $K$ is a subgroup of $U$, but the proof of this fact can be easily generalized to the case of a Lie group homomorphism $K\rightarrow U$.

We use this theorem for example when $E = \got{p}$. Making the identification $\got{g}^*\cong\got{g}$ induced by the Killing form on $\got{g}$, we get $\Phi_\got{p}(X) = -[X,[z_0,X]]$, for all $X\in\got{p}$, and then $\langle\Phi_\got{p}(X),z_0\rangle = \|[z_0,X]\|^2 >0$, if $X\neq 0$, where $\|\cdot\|$ is the norm  on $\got{p}$ induced by $B_{\got{g}}$.

Let $\Lambda\in\got{t}^*_+$. As for $K\cdot\Lambda\times\got{p}$, the manifold $K\cdot\Lambda\times E$ carries a canonical symplectic product $\Omega_{K\cdot\Lambda}\oplus\Omega_{E}$ given by the direct product of the symplectic manifolds $(K\cdot\Lambda,\Omega_{K\cdot\Lambda})$ and $(E,\Omega_E)$. Moreover, the diagonal action of $K$ on $(K\cdot\Lambda\times E,\Omega_{K\cdot\Lambda}\oplus\Omega_{E})$ is Hamiltonian, with moment map
\[
\begin{array}{cccl}
\Phi_{K\cdot\Lambda\times E} : & K\cdot\Lambda\times E & \rightarrow & \got{k}^* \\
& (k\Lambda, v) & \mapsto & k\Lambda + \Phi_E(v).
\end{array}
\]

Now, assume that $\Phi_E$ is proper. By \cite[Theorem 1.5]{paradan}, since $(E,\Omega_E)$ is a symplectic vector space endowed with a proper Hamiltonian $K$-action, then the moment map $\Phi_{K\cdot\Lambda\times E}:K\cdot\Lambda\times E\rightarrow\got{k}^*$ is also proper. We will denote by $\Delta_K(K\cdot\Lambda\times E):=\Phi_{K\cdot\Lambda\times E}(K\cdot\Lambda\times E)\cap\got{t}_+^*$ the associated Kirwan polyhedron.

Actually, we are going to see that the polyhedral convex set $\Delta_K(K\cdot\Lambda\times E)$ is an affine section of some bigger polyhedral set.

It is a well-known fact that the coadjoint orbit $K\cdot\Lambda$ is a symplectic quotient of the cotangent bundle $T^*K$ relatively to the right multiplication of $K$. Thus, $K\cdot\Lambda\times E$ is a symplectic quotient of $T^*K\times E$. Let us clarify this.

First, we recall that the cotangent bundle $T^*K$ is identified to $K\times\got{k}^*$ by means of left translations. The group $K\times K$ acts on $K\times\got{k}^*\times E$ by
\[
(k_1,k_2)\cdot(k,\mu,v) := (k_1kk_2^{-1},k_2\mu,k_2v),
\]
for all $(k_1,k_2)\in K\times K$ and $(k,\mu,v)\in K\times\got{k}^*\times E$. This action is Hamiltonian, with moment map
\[
\Phi_{K\times K}(k,\mu,v) = (l\mu,-\mu+\Phi_E(v)), \quad \forall (k,\mu,v)\in K\times\got{k}^*\times E,
\]
see for instance \cite{guillemin_sternberg, sjamaar_convexity}. This induces a Hamiltonian action of $K\times\{1\}$ on $K\times\got{k}^*\times E$, with moment map
\[
\Phi_{K\times\{1\}}(k,\mu,v) = l\mu \in\got{k}^*, \quad \forall (k,\mu,v)\in K\times\got{k}^*\times E,
\]
and, clearly, the symplectic $K$-manifold $K\cdot\Lambda\times E$ can be identified to the symplectic quotient $\Phi_{K\times\{1\}}^{-1}(K\cdot\Lambda)/(K\times\{1\})$. Moreover, for any $(k,v)\in K\times E$, we have
\[
\Phi_{K\cdot\Lambda\times E}(k\Lambda,v) = k\Lambda+\Phi_E(v) = \Phi_{\{1\}\times K}(l,k(-\Lambda),v)
\]
for any $l\in K$. Thus, taking $l=w_0k^{-1}$ yields
\[
\Phi_{K\times K}(w_0k^{-1},k(-\Lambda),v) = (-w_0\Lambda,\Phi_{K\cdot\Lambda\times E}(k\Lambda,v)),
\]
which obviously proves next proposition.

\begin{prop}
\label{prop:relation_between_the_two_momentpolyhedron}
Let $\Lambda$ and $\mu$ be in $\got{t}_+^*$. Then $\mu\in\Delta_K(K\cdot\Lambda\times E)$ if and only if $(-w_0\Lambda,\mu)\in\Delta_{K\times K}(T^*K\times E)$.
\end{prop}

\begin{rem}
One can show that the moment map $\Phi_{K\times K}$ is proper, because $\Phi_E$ is. Hence, by the noncompact version of Kirwan's Hamiltonian convexity theorem \cite{lerman_et_al,BS00}, the set $\Delta_{K\times K}(T^*K\times E)$ is convex locally polyhedral. So, it boils down to compute the equations of $\Delta_{K\times K}(T^*K\times E)$.
\end{rem}

\begin{nota}
We will denote by $\wedge^*$ the weight lattice of $\got{t}^*$. It is the set of elements $\frac{1}{i}\alpha$, where $\alpha$ is the differential of a character of $T$.

We will also denote by $\wedge^*_{\Q}=\wedge^*\otimes_{\Z}\Q$ (resp. $\wedge_+^*=\wedge^*\cap\got{t}_+^*$, resp. $\wedge^*_{\Q,+}=\wedge^*_{\Q}\cap\got{t}_+^*$) the set of rational weights (resp. dominant weights, resp. dominant rational weights) of $\got{t}^*$.
\end{nota}

\medskip

From \cite{lerman_et_al}, it turns out that $\Delta_{K\times K}(T^*K\times E)$ is a rational locally polyhedral convex set, because $\Phi_{K\times K}$ is proper. Then, determining the equations of $\Delta_{K\times K}(T^*K\times E)$ in $\got{t}^*$ is then equivalent to determining the ones of its set of rational points in $\wedge^*_{\Q,+}$.

From now on, for any dominant weight $\nu\in\wedge^*_+$, $V^K_{\nu}$ denotes the irreducible representation of $K$ with highest weight $\nu$.

\begin{thm}
\label{thm:momentpolyhedron_and_irreduciblerepresentations}
Assume $\Phi_E: E\rightarrow \got{k}^*$ is proper. Let $(\mu,\nu)\in(\wedge^*_{\Q,+})^2$. Then $(\mu,\nu)\in\Delta_{K\times K}(T^*K\times E)$ if and only if there exists an integer $N\geq 1$ such that $(N\mu,N\nu)\in(\wedge^{*})^2$ and $\left(V^{K*}_{N\mu}\otimes V^{K*}_{N\nu}\otimes \C[E]\right)^K \neq 0$.
\end{thm}

\begin{proof}
Since $T^*K$ is $K\times K$-isomorphic to $\KC$, the complexified group of the compact Lie group $K$, then $T^*K\times E$ is an affine variety. Theorem 4.9 of \cite{sjamaar_convexity} yields that $\Delta_{K\times K}(T^*K\times E)$ is the convex cone generated by the monoid
\[
\left\{(\mu,\nu)\in(\wedge_+^*)^2\ |\ V_{(\mu,\nu)}^{K\times K}\subseteq \C[\KC\times E]\right\}.
\]
By Frobenius' theorem, we have $\C[\KC] = \bigoplus_{\delta\in\wedge_+^*}V^K_{\delta}\otimes V_{\delta}^{K*}$. Thus, a pair $(\mu,\nu)\in(\wedge^*_+)^2$ satisfies $V_{(\mu,\nu)}^{K\times K}\subseteq \C[\KC\times E]$ if and only if $V^K_{\mu}\subseteq V_{\nu}^{K*}\otimes\C[E]$, or, equivalently, $\left(V^{K*}_{\mu}\otimes V^{K*}_{\nu}\otimes \C[E]\right)^K \neq 0$. Hence, Theorem \ref{thm:momentpolyhedron_and_irreduciblerepresentations} directly follows.
\end{proof}

\begin{defn}
\label{defn:DGIT}
When $\Phi_E: E\rightarrow \got{k}^*$ is proper, we define the set
\[
\Pi_{\Q}(E) = 
\left\{(\mu,\nu)\in(\wedge^*_{\Q,+})^2 \left|
\begin{array}{l}
\exists N\in\Z_{>0} \mbox{ such that } (N\mu,N\nu)\in(\wedge^*)^2, \\
\mbox{and } \left(V^{K*}_{N\mu}\otimes V^{K*}_{N\nu}\otimes \C[E]\right)^K \neq 0
\end{array}\right.\right\}.
\]
\end{defn}

\begin{cor}
\label{cor:DeltaClassic_DeltaGIT}
Assume $\Phi_E: E\rightarrow \got{k}^*$ is proper. Then, $\Delta_{K\times K}(T^*K\times E)\cap(\wedge^*_{\Q})^2 = \Pi_{\Q}(E)$, and $\Delta_{K\times K}(T^*K\times E)$ is the closure of $\Pi_{\Q}(E)$ in $\got{t}^*$. In particular, $\Delta_{K\times K}(T^*K\times E)$ is a rational polyhedral convex cone.
\end{cor}

\begin{proof}
By definition of $\Pi_{\Q}(E)$, this is given by Theorem \ref{thm:momentpolyhedron_and_irreduciblerepresentations}. We point out that the second assertion is also a consequence of Theorem 4.9 from \cite{sjamaar_convexity}.
\end{proof}



\section{Semiample cone and well covering pairs}

We introduce in this section the main GIT notions that we will use in the rest of the article: the semiample cone, in the sense of Dolgachev-Hu \cite{Dolgachev-Hu}, which is a polyhedral convex cone whose equations are known thanks to Ressayre's well covering pairs \cite{ressayre10}.


\subsection{Notations}
\label{subsection:notations_preliminaries_GIT}

In this section, $K$ will denote a compact connected real Lie group, and $\KC$ its complexification. Let $X$ be a projective $\KC$-variety, and let $\PicG(X)$ denote the group of isomorphism classes of $\KC$-linearized line bundles on $X$. Recall that a $\KC$-linearized line bundle is a line bundle $\mathcal{L}$ on $X$ together with a lifting of the $\KC$-action to $\mathcal{L}$ which is linear on the fibers. The group structure is given by the tensor product. See \cite{dolgachev03, knopkraftvurst} for more details about $\PicG(X)$.

Let $\PicGQ(X) := \PicG(X)\otimes_{\Z}\Q$ be the $\Q$-vector space generated by the elements of $\PicG(X)$, and, for any $\mathcal{L}\in\PicG(X)$, $\mathrm{H}^0(X,\mathcal{L})$ the $\KC$-module of regular sections of $\mathcal{L}$.

Let $T$ be a maximal torus of $K$. We fix $\TC$ a maximal torus of $\KC$ such that $T=K\cap \TC$, and $B$ a Borel subgroup of $\KC$ containing $\TC$. From now on, we will identify the group of characters of $\TC$ with the weight lattice $\wedge^*\subset\got{t}^*$. For any dominant weight $\mu\in\wedge^*$, we will denote by $V_{\mu}^{\KC}$ the irreducible representation  of either $K$ or $\KC$, with highest weight $\mu$.

\begin{nota}
For $M$ a complex representation of $\KC$, we will denote by $X_M$ the smooth $\KC$-variety
\[
X_M := \KC/B\times \KC/B\times \mathbb{P}(M),
\]
equipped with the diagonal action of $\KC$.
\end{nota}
%

\subsection{Semistability}

For any $\KC$-linearized line bundle $\mathcal{L}$ on $X$, one of the most important GIT objects associated to $\mathcal{L}$ and $X$ is the set of semistable points on $X$,
\[
\XssL = \left\{x\in X\ |\ \exists k\geq 1, \exists s\in \mathrm{H}^0(X,\mathcal{L}^{\otimes k})^{\KC}, \mbox{ s.t. } s(x)\neq 0\right\}.
\]
This is not the standard definition of $\Xss(\mathcal{L})$, but it is if $\mathcal{L}$ is ample. This definition is introduced in \cite{ressayre10}. The reader may refer to \cite{dolgachev03} for the standard definition.

If $\mathcal{L}$ is ample, we have a categorical quotient $\pi : \XssL \rightarrow \XssL/\!/\KC$, such that $\XssL/\!/\KC$ is a projective variety and $\pi$ is affine.

It is clear that for any $\KC$-linearized $\mathcal{L}$, and any positive integer $n$, we have $\Xss(\mathcal{L}) = \Xss(\mathcal{L}^{\otimes n})$. So we can define $\Xss(\mathcal{L})$ for any element $\mathcal{L}\in\PicGQ(X)$.

\subsection{Semiample cone}

Now we are going to introduce the notion of  semiample cone $C_{\Q}(X)^+$ associated to an irreducible projective variety.
%

We denote by $\PicG(X)^+$ the set of semiample $\KC$-linearized line bundle on $X$. Furthermore, we denote by $\PicGQ(X)^+$ the convex cone generated by the semiample elements of $\PicGQ(X)$. Now, we define the \emph{semiample cone}
\[
C_{\Q}(X)^+ = \{\mathcal{L}\in\PicGQ(X)^+\ | \ \Xss(\mathcal{L}) \neq\emptyset\}.
\]

%

In our setting, we will have some special type of variety $X$. Let $\hKC$ be a connected reductive group such that $\KC\subseteq\hKC$, and $Q$ (resp. $\wh{Q}$) be a parabolic subgroup of $\KC$ (resp. $\hKC$). The next theorem justify the terminology of semiample cone.

\begin{thm}[\cite{ressayre10}, Proposition 10]
\label{thm:sacX_closedconvpoly_and_closureofacX}
When $X=\KC/Q\times\hKC/\wh{Q}$, then the semiample cone $C_{\Q}(X)^+$ is a closed convex polyhedral cone in $\PicGQ(X)$. Moreover, if $C_{\Q}(X)^{+}$ contains an ample line bundle, then $C_{\Q}(X)^+$ is the closure, in $\PicGQ(X)$, of the set of its ample elements.
\end{thm}

\subsection{Hilbert-Mumford numerical criterion}
\label{subsection:hilbertmumford_numcriterion}

Let $\mathcal{L}\in\PicG(X)$ be a semiample $\KC$-linearized line bundle on $X$. For any point $x\in X$ and one parameter subgroup $\lambda$ of $\KC$, one can define the limit $\lim_{t\rightarrow 0} \lambda(t)\cdot x$ as follows: since $X$ is projective, we can extend the rational morphism $f:t\in\C^*\mapsto \lambda(t)\cdot x\in X$ to an algebraic map $\tilde{f}:\mathbb{P}^1\rightarrow X$. Let $\lim_{t\rightarrow 0}\lambda(t)\cdot x = \tilde{f}(0)$ be this limit.

Let $z = \lim_{t\rightarrow 0}\lambda(t)\cdot x$, this point $z$ is then a fixed point of the induced action of $\C^*$ by $\lambda$ on $X$. Since $\mathcal{L}$ is $\KC$-linearized, the group $\C^*$ acts linearly on the fiber $\mathcal{L}_z$ over the point $z$. This action defines $\mu^{\mathcal{L}}(x,\lambda)\in\Z$, setting $\lambda(t)\cdot v = t^{-\mu^{\mathcal{L}}(x,\lambda)}v$, for any $t\in\C^*$ and $v\in\mathcal{L}_z$.

It is easy to check that the numbers $\mu^{\mathcal{L}}(x,\lambda)$ satisfy the three following properties:
\begin{enumerate}
\item $\mu^{\mathcal{L}}(g.x,g\cdot\lambda\cdot g^{-1}) = \mu^{\mathcal{L}}(x,\lambda)$, for all $g\in \KC$,
\item the map $\mathcal{L}\in\PicG(X)\mapsto\mu^{\mathcal{L}}(x,\lambda)\in\Z$ is a group homomorphism.
\item for all $n\in\Z_{\geq 0}$, we have $\mu^{\mathcal{L}}(x,n\lambda) = n\mu^{\mathcal{L}}(x,\lambda)$,
\end{enumerate}
where $n\lambda$ is the one parameter subgroup of $\KC$ defined by $(n\lambda)(t) := \lambda(t^n)$, for all $t\in\C^*$.

\begin{defn}
A one parameter subgroup $\lambda$ of $\KC$ is \emph{indivisible}, if for any one parameter subgroup $\lambda'$ of $\KC$ and integer $n>1$, $n\lambda'$ is not equal to $\lambda$.
\end{defn}

These numbers $\mu^{\mathcal{L}}(x,\lambda)$ give a characterization of the semistable point set of $X$ for $\mathcal{L}$. Indeed, by \cite{mumford94} in the ample case, and \cite[Lemma $2$]{ressayre10} in the semiample case,
\[
x\in\Xss(\mathcal{L}) \ \Longleftrightarrow \ \mu^{\mathcal{L}}(x,\lambda) \leq 0, \ \mbox{for any one parameter subgroup } \lambda \mbox{ of } \KC.
\]
We notice we could only consider indivisible one parameter subgroups in the above equivalence, using ($3$).

To any one parameter subgroup $\lambda$ of $\KC$, we can associate a parabolic subgroup
\[
P(\lambda) = \left\{g\in \KC\ | \ \lim_{t\rightarrow 0} \lambda(t)\cdot g\cdot \lambda(t)^{-1} \mbox{ exists in } \KC\right\}.
\]
Let $\KC^{\lambda}$ be the centralizer of the image of $\lambda$ in $\KC$. Actually, this group $\KC^{\lambda}$ is a Levi subgroup of $P(\lambda)$. 




The next proposition is a well-known and easy fact.

\begin{prop}
\label{prop:product_twobundles_numericalcriterion}
If $(X_1,\mathcal{L}_1)$ and $(X_2,\mathcal{L}_2)$ are two $\KC$-linearized varieties, then
\[
\mu^{\mathcal{L}_1\boxtimes\mathcal{L}_2}((x_1,x_2),\lambda) = \mu^{\mathcal{L}_1}(x_1,\lambda) + \mu^{\mathcal{L}_2}(x_2,\lambda),
\]
for all $(x_1,x_2)\in X_1\times X_2$.
\end{prop}

\subsection{Well covering pairs}
\label{subsection:redgp_git_wellcoveringpairs}

The notion of well covering pair was introduced by Ressayre in \cite{ressayre10}. This gives a necessary and sufficient condition for $\mathcal{L}\in\PicG(X)^+$ being in the semiample cone $C_{\Q}(X)^+$, in terms of linear equations.

Let $X$ be a smooth projective variety. In the rest of this paper, if $\lambda$ is a one parameter subgroup of $\KC$, we denote by $X^{\lambda}$ the set of points of $X$ fixed by the action of the subgroup $\lambda(\C^*)$ in $\KC$.

\begin{defn}
Let $\lambda$ be a one parameter subgroup of $\KC$ and $C$ a irreducible component of $X^{\lambda}$. Let $C^+ := \{x\in X\ ; \ \lim_{t\rightarrow 0} \lambda(t)\cdot x \in C\}$. We consider the following $\KC$-equivariant map :
\[
\begin{array}{cccc}
\eta : & \KC\times_{P(\lambda)} C^+ & \longrightarrow & X \\
& [g,x] & \longmapsto & g\cdot x.
\end{array}
\]
The pair $(C,\lambda)$ is said \emph{covering} (resp. \emph{dominant}) if $\eta$ is a birational map (resp. dominant map). The pair $(C,\lambda)$ is said \emph{well covering} if it is a covering pair and if there exists a $P(\lambda)$-stable open subset $\Omega$ of $C^+$ intersecting $C$ such that $\eta$ induces an isomorphism from $\KC\times_{P(\lambda)}\Omega$ onto an open subset of $X$.
\end{defn}

Let us fix a line bundle $\mathcal{L}\in\PicG(X)$ and a one parameter subgroup $\lambda$ of $\KC$. Then the map $x\mapsto\mu^{\mathcal{L}}(x,\lambda)$ takes its values in $\Z$. We notice that, for any irreducible component $C$ of $X^{\lambda}$, the value of $\mu^{\mathcal{L}}(x,\lambda)$ does not depend of $x\in C$, thus we can define the integer $\mu^{\mathcal{L}}(C,\lambda)$.

The two following results of Ressayre asserts that the integers $\mu^{\mathcal{L}}(C,\lambda)$ associated to well covering pairs (resp. dominant pairs) of $X$ give a complete description of $C_{\Q}(X)^+$.

\begin{lem}[\cite{ressayre10}, Lemma 3]
\label{lem:CN_dominantpairs_semiamplecone}
Let $(C,\lambda)$ be a dominant pair of $X$ and $\mathcal{L}\in C_{\Q}(X)^+$. Then $\mu^{\mathcal{L}}(C,\lambda)\leq 0$.
\end{lem}

\begin{thm}[\cite{ressayre10}, Proposition 4]
\label{thm:ressayre_equations_amplecone}
Assume $X$ is a smooth projective variety. Then the semiample cone $C_{\Q}(X)^+$ is the set of the line bundles $\mathcal{L}\in\PicGQ(X)^+$ such that for all well covering pairs $(C,\lambda)$ of $X$, we have $\mu^{\mathcal{L}}(C,\lambda) \leq 0$.
\end{thm}

In some special cases of varieties $X$, including the varieties of type $X_M$, we obtain a smaller set of equations determining the semiample cone.


\begin{thm}[\cite{ressayre10}, Theorem 3]
\label{thm:coneample_equations_stabilizer}
Let $X=\KC/B\times\hKC/\wh{Q}$. Assume $C_{\Q}(X)^{+}$ has nonempty interior in $\PicGQ(X)$. Let $\mathcal{L}\in\PicGQ(X)^{+}$. Then $\mathcal{L}$ is in $C_{\Q}(X)^{+}$ if and only if for all well covering pairs $(C,\lambda)$ of $X$ such that there exists $x\in C$ with $(\KC)_x^{\circ} = \lambda(\C^*)$, we have $\mu^{\mathcal{L}}(C,\lambda)\leq 0$.
\end{thm}

\begin{rem}
\label{rem:ops_onlyindivisibledominantneeded}
Noticing that, for any one parameter subgroup $\lambda$ of $\KC$ and $g\in \KC$, the set of fixed points of $g\lambda g^{-1}$ is $X^{g\lambda g^{-1}}=g\cdot X^{\lambda}$, and $C$ is a irreducible component of $X^{\lambda}$ if and only if $g\cdot C$ is a irreducible component of $X^{g\lambda g^{-1}}$, we can apply ($1$) and ($3$) of subsection \ref{subsection:hilbertmumford_numcriterion} so as to show that it is sufficient to consider only the well covering pairs $(C,\lambda)$ with $\lambda$ dominant indivisible one parameter subgroup of $\TC$, in the statements of Theorems \ref{thm:ressayre_equations_amplecone} and \ref{thm:coneample_equations_stabilizer}.
\end{rem}

\begin{rem}
\label{rem:cns_with_dominantpairs}
It is also clear from Lemma \ref{lem:CN_dominantpairs_semiamplecone} that Theorems \ref{thm:ressayre_equations_amplecone} and \ref{thm:coneample_equations_stabilizer} are also true if we replace ``well covering pairs" by ``dominant pairs" or ``covering pairs" in their statements.
\end{rem}

\subsection{Description of the irreducible components of $X_M^{\lambda}$}

If $\beta$ is a weight of $\TC$ in $M$, we define the associated weight space
\[
M_{\beta} := \{v\in M\,\ | \ d\rho(Y)v = \beta(Y)v, \mbox{ for all } Y\in \got{t}_{\C}\},
\]
and, for all $k\in\Z$, and for all one parameter subgroup $\lambda$ of $\KC$,
\[
M_{\lambda,k} := \{v\in M\ | \ \lambda(t)\cdot v = t^k v, \forall t\in\C^*\}.
\]
We notice that $M_{\lambda,0} = M^{\lambda}$ is the subspace of vectors of $M$ fixed by $\lambda$.

Now, let us fix a dominant one parameter subgroup $\lambda$ of $\TC$. We will denote by $W=W(\TC;\KC)$ the Weyl group of $\TC$ in $\KC$ and $W_{\lambda}$ the Weyl group of the Levi subgroup $\KC^{\lambda}$ of $P(\lambda)$.

It is clear that $X_M^{\lambda} = (\KC/B)^{\lambda}\times (\KC/B)^{\lambda} \times \mathbb{P}(M)^{\lambda}$. The first two factors are of the form $(\KC/B)^{\lambda} = \bigcup_{w\in W_{\lambda}\backslash W} \KC^{\lambda}wB/B$. Obviously, we will also have $\mathbb{P}(M)^{\lambda} = \bigcup_{m\in\Z} C_m$, where $C_m = \mathbb{P}(M_{\lambda,m})$ for any $m\in\Z$. For $(w,w',m)\in W/W_{\lambda}\times W/W_{\lambda}\times\Z$, we define
\[
C(w,w',m) = \KC^{\lambda}w^{-1}B/B \times \KC^{\lambda}w'^{-1}B/B \times C_m.
\]
We keep the $^{-1}$ introduced in the notations of \cite{ressayre10}. We now have
\[
X_M^{\lambda} = \bigcup_{\stackrel{w,w'\in W/W_{\lambda}}{m\in\Z}} C(w,w',m).
\]

\section{Semiample cone of the $\KC$-variety $X_{E\oplus\C}$}

As we are interested in computing the equations of the convex polyhedron $\Delta_K(K\cdot\Lambda\times E)$ defined in subsection \ref{subsection:Kirwanpolyhedron_KLambdaE}, we would like to apply GIT methods to it, that is applying the well covering pair's machinery to the ``semimaple cone of $K\cdot\Lambda\times E$''. Here are two issues: $K\cdot\Lambda\times E$ is neither a flag variety nor a projective variety. We get around this two problems simultaneously by considering its compactification $K\cdot\Lambda\times \mathbb{P}(E\oplus\C)$.

Indeed, we are going to see that the semiample cone of $X_{E\oplus\C} = \KC/B\times\KC/B\times\mathbb{P}(E\oplus\C)$ is closely related to the polyhedral cone $\Delta_{K\times K}(T^*K\times E)$, and thus to $\Delta_K(K\cdot\Lambda\times E)$.

\subsection{The Picard group $\PicG(X_M)$}

We begin with describing the set of $\KC$-linearized line bundles on the variety $X_M:=\KC/B\times \KC/B\times\mathbb{P}(M)$, where $M$ is any complex representation of $\KC$.

First, we consider the complete flag variety $\KC/B$. It is a well known fact that the $\KC$-linearized line bundles on $\KC/B$ are bijectively identified with the elements of the weight lattice $\wedge^*$ of $\KC$, by the map
\[
\mu\in\wedge^*\mapsto\mathcal{L}_{\mu} := \KC \times_{B} \C_{-\mu}\in\PicG(\KC/B),
\]
with $\C_{-\mu}$ the representation of $B$ on $\C$ with weight $-\mu$. Moreover, the Borel-Weyl theorem implies that the semiample line bundles are those associated to dominant weights, that is $\PicG(\KC/B)^+ \cong \wedge^*_+$.


Let us now look at the last factor in $X$, the projective space $\mathbb{P}(M)$. Let $\mathcal{X}(\KC)$ denote the group of characters of $\KC$, that is, the group of rational homomorphisms of algebraic groups from $\KC$ to $\mathbb{G}_m$. We know that $\Pic(\mathbb{P}(M))$ is isomorphic to $\Z$, by the map $k\in\Z\mapsto\mathcal{L}_k\in\Pic(\mathbb{P}(M))$, with $\mathcal{L}_k = M\backslash\{0\}\times_{\C^*}\C_{k}$, where $\C_{k}$, for any $k\in\Z$, is the complex vector space $\C$ equipped with the action of $\C^*$ defined by
\[
\forall v\in\C,\ \forall z\in\C^*,\ z\cdot v := z^k v.
\]
The semiample line bundles of $\Pic(\mathbb{P}(M))$ are the line bundles $\mathcal{L}_k$, with $k\in\Z_{\geq 0}$. 

Now, for any $(k,\chi)\in\Z\times\mathcal{X}(\KC)$, the action of $\KC$ on $\mathcal{L}_k$ defined by
\[
g\cdot[x,z] := [g\cdot x, \chi(g^{-1})z], \text{ for all } g\in \KC, \text{ and all } [x,z]\in\mathcal{L}_k,
\]
is a $\KC$-linearization of the line bundle $\mathcal{L}_k$, and we will denote by $\mathcal{L}_{\chi,k}$ this $\KC$-linearized line bundle of $\mathbb{P}(M)$. Consequently, the forgetful homomorphism $\alpha:\PicG(\mathbb{P}(M))\rightarrow\Pic(\mathbb{P}(M))$ is surjective. By \cite[Lemma 2.2]{knopkraftvurst} (or \cite[Theorem 7.1]{dolgachev03}), $\ker\alpha$ is isomorphic to $\mathcal{X}(\KC)$. Thus any $\KC$-linearized line bundle $\mathcal{L}$ is isomorphic to $\mathcal{L}_{\chi,k}$ for some $(\chi,k)\in\mathcal{X}(\KC)\times\Z$. Clearly, for all $(\chi,k)\in\mathcal{X}(\KC)\times\Z$, $\mathcal{L}_{\chi,k}$ is isomorphic to the tensor product of $\KC$-linearized line bundles $\mathcal{L}_k\boxtimes(\mathbb{P}(M)\times\C_{-\chi})$.

We identify $\mathcal{X}(\KC)$ with the sublattice $(\got{k}^*)^{K}\cap\wedge^*$ of $\wedge^*$. We denote by $\mathcal{X}(\KC)_{\Q}$ the $\Q$-vector space generated by $\mathcal{X}(\KC)$ in $\got{k}^*$. We notice that $(\got{k}^*)^{K}$ is equal to the dual $\got{z}(\got{k})^*:=[\got{k},\got{k}]^{\bot}$ of the center of $\got{k}$.

The following proposition is analogous to \cite[Lemma 13]{ressayre10}.

\begin{prop}
The $\Q$-linear map
\begin{equation}
\label{eq:application_Qlin_surjective_onto_PicKC(X_M)}
\begin{array}{rrcl}
\pi_{M}:&\wedge^*_{\Q}\times\wedge^*_{\Q}\times\Q & \longrightarrow & \PicG(X_M)_{\Q} \\
&(\mu,\nu,r) & \longmapsto & \mathcal{L}_{\mu,\nu,r} := \mathcal{L}_{\mu}\boxtimes\mathcal{L}_{\nu}\boxtimes\mathcal{L}_{r}
\end{array}
\end{equation}
is surjective. Its kernel is the subspace $\{(\mu,-\mu,0)\ |\ \mu\in\mathcal{X}(\KC)_{\Q}\}$.
\end{prop}

\begin{proof}
Let $\mu$ and $\nu$ be two characters of $\KC$. Then, as elements of $\Pic(\KC/B)$, the line bundles $\mathcal{L}_{\mu}$ and $\mathcal{L}_{\nu}$ are trivial. That is, $\mathcal{L}_{\mu,\nu,0}$ is a $\KC$-linearization of the trivial bundle of $\Pic(X_M)$ and $\mathcal{L}_{\mu,\nu,0} = X_M\times\C_{-(\mu+\nu)}$. In particular, any $\KC$-linearization of the trivial bundle on $X_M$ is in the image of \eqref{eq:application_Qlin_surjective_onto_PicKC(X_M)}.

Now, let $\mathcal{L}\in\PicG(X_M)$. Let $\mathcal{L}'$ denote the line bundle obtained by forgetting the $\KC$-action on $\mathcal{L}$. From \cite[Proposition 1.5]{mumford94} or \cite[Corollaire 7.2]{dolgachev03}, we know that there exists a positive integer $n$ and a $\KC\times\KC\times GL_{\C}(M)$-linearization $\mathcal{M}$ of $\mathcal{L}'^{\otimes n}$. But $X_M$ is a flag manifold of the reductive group $\KC\times\KC\times GL_{\C}(M)$, then $\mathcal{M}\cong\mathcal{L}_{\mu}\boxtimes\mathcal{L}_{\nu}\boxtimes\mathcal{L}_{\hat{\nu}}$, with $\mu,\nu\in\wedge^*_+$ and $\hat{\nu}$ a dominant weight of $GL_{\C}(M)$. Furthermore, the $GL_{\C}(M)$-action on $\mathcal{L}_{\hat{\nu}}$ induces a $\KC$-linearization of the line bundle $\mathcal{L}_{\hat{\nu}}$, that is, $\mathcal{L}_{\hat{\nu}}\cong\mathcal{L}_{\chi,k}\in\PicG(\mathbb{P}(M))$ for some $k\in\Z$ and $\chi\in\mathcal{X}(\KC)$. Finally, $\mathcal{M}\cong\mathcal{L}_{\mu}\boxtimes\mathcal{L}_{\nu}\boxtimes\mathcal{L}_{\chi,k}\cong \mathcal{L}_{\mu+\chi,\nu,k}$ is a $\KC$-linearization of $\mathcal{L}'^{\otimes n}$ and $\mathcal{L}_{\mu+\chi,\nu,k}^{\otimes(-1)}\otimes\mathcal{L}^{\otimes n}$ is a $\KC$-linearization of the trivial bundle. From the first paragraph of this proof, it means that $\mathcal{L}$ is in the image of \eqref{eq:application_Qlin_surjective_onto_PicKC(X_M)}.

Assume that $\mathcal{L}_{\mu,\nu,k}$ is trivial. Since $X_M$ is a flag variety, one can check that, necessarily, $\mathcal{L}_{\mu}$, $\mathcal{L}_{\nu}$ and $\mathcal{L}_{k}$ are trivial if we don't take account of the $\KC$-actions. Consequently, $\mu$ and $\nu$ are characters of $\KC$, $k=0$, and $\mathcal{L}_{\mu,\nu,k} = X_M\times\C_{-(\mu+\nu)}$. Finally, we have $\mu+\nu=0$ because $\mathcal{L}_{\mu,\nu,k}$ is trivial in $\PicG(X_M)$.
\end{proof}

\begin{rem}
\label{rem:pi_M^-1(CQ(XM))_is_closedpolyhedralcone}
By Theorem \ref{thm:sacX_closedconvpoly_and_closureofacX}, we know that $C_{\Q}(X_M)^+$ is a closed convex polyhedral cone of $\PicG(X_M)_{\Q}$. Since $\pi_{M}$ is a linear projection between finite dimensional vector spaces, then $\pi_{M}^{-1}(C_{\Q}(X_{M})^+)$ is a closed convex polyhedral cone of $(\wedge^*)^2$. Moreover, its equations are described by Theorem \ref{thm:ressayre_equations_amplecone}.
\end{rem}

\begin{prop}
\label{prop:equations_of_CQ(XM)}
Let $(\mu,\nu,r)$ be in $\wedge^*_{\Q,+}\times\wedge^*_{\Q,+}\times\Q_{\geq 0}$. Then $\mathcal{L}_{\mu,\nu,r}\in C_{\Q}(X_M)^+$ (or, equivalently, $(\mu,\nu,r)\in\pi_{M}^{-1}(C_{\Q}(X_M)^+)$) if and only if
\begin{equation}
\label{eq:equation_of_CQ(XM)_fromgivenwellcoveringpair}
\langle w\lambda,\mu\rangle + \langle w'\lambda,\nu\rangle + mr \leq 0,
\end{equation}
for all indivisible dominant one parameter subgroups $\lambda$ of $\TC$, and for all $(w,w',m)$ in $W/W_{\lambda}\times W/W_{\lambda}\times\Z$ such that $(C(w,w',m),\lambda)$ is a well covering pair (resp. dominant pair) of $X_M$.
\end{prop}

The form of equation \eqref{eq:equation_of_CQ(XM)_fromgivenwellcoveringpair} follows from Proposition \ref{prop:product_twobundles_numericalcriterion}, and from the next two lemmas, which proofs are straightforward calculations. Here, we assume that $\lambda$ is a one parameter subgroup of $\TC$.

\begin{lem}
\label{lem:numericalcriterion_GoverB}
Let $X=\KC/B$, and let $w$ be an element of the Weyl group of $\TC$ in $\KC$, and $\nu$ a dominant weight of $B$. Then $\mu^{\mathcal{L}_{\nu}}(w^{-1}B/B, \lambda) = \langle w\lambda,\nu\rangle$.
\end{lem}
%
%

\begin{lem}
\label{lem:numericalcriterion_P(M)}
Let $M$ be a complex representation of $\KC$, $M\neq\{0\}$, and $X=\mathbb{P}(M)$. Let $m\in\Z$, and $v\in M_{\lambda,m}\setminus\{0\}$. 
Then, for all $k\in\Z$, we have
\[
\mu^{\mathcal{L}_{k}}([v],\lambda) = mk,
\]
where $[v]$ is the class of $v$ in $\mathbb{P}(M)$.
\end{lem}




\subsection{Projection of the semiample cone of $X_{E\oplus\C}$}

Let $E$ be a complex representation of $\KC$. We denote by $\zeta:\KC\rightarrow GL_{\C}(E)$ the homomorphism of reductive groups associated to the representation $E$. We consider the induced representation $E\oplus\C$ of $\KC$, where $\KC$ acts trivially on $\C$.

\begin{prop}
\label{prop:SAC_equivalentdef_with_irrep}
Let $(\mu,\nu,r)\in\wedge^*_{\Q,+}\times\wedge^*_{\Q,+}\times\Q_{\geq 0}$, and let $\mathrm{pr}:\wedge^*_{\Q}\times\wedge^*_{\Q}\times\Q\rightarrow\wedge^*_{\Q}\times\wedge^*_{\Q}$ be the canonical linear projection.
\begin{enumerate}
\item We have $\mathcal{L}_{\mu,\nu,r}\in C_{\Q}(\XEC)^+$ if and only if there exists $k\geq 1$ such that $(k\mu,k\nu)\in(\wedge^*)^2$, $kr\in\Z$, and $\left(V_{k\mu}^{\KC*}\otimes V_{k\nu}^{\KC*}\otimes \C_{\leq kr}[E]\right)^{\KC}\neq 0$.
\item The set $C_{\Q}(\XEC)^+$ contains an ample line bundle.
\item We have $\Pi_{\Q}(E) = \mathrm{pr}(\pi_{E\oplus\C}^{-1}(C_{\Q}(X_{E\oplus\C})^+))$.
\end{enumerate}
\end{prop}

\begin{proof}
By definition, for any $(\mu,\nu,r)\in\wedge^*_{\Q,+}\times\wedge^*_{\Q,+}\times\Q_{\geq 0}$, the line bundle $\mathcal{L}_{\mu,\nu,r}$ is in the semiample cone $C_{\Q}(\XEC)^+$ if and only if there exists a positive integer $k$ such that $(k\mu,k\nu,kr)\in\wedge^*\times\wedge^*\times\Z$, and $\Xss(\mathcal{L}_{k\mu,k\nu,kr}) \neq \emptyset$. This last assertion is equivalent to $\mathrm{H}^0(X,\mathcal{L}_{k\mu,k\nu,kr}^{\otimes k'})^{\KC}$ containing a nontrivial element, for some $k'\geq 1$. Multiplying $k$ by $k'$, we may assume that $\mathrm{H}^0(X,\mathcal{L}_{k\mu,k\nu,kr})^{\KC}$ is not zero. The $\KC$-module $\mathrm{H}^0\left(\mathbb{P}(E\oplus\C),\mathcal{L}_{kr}\right)$ is isomorphic to the $\KC$-module of polynomials on $E$ of degree less than or equal to $kr$, denoted by $\C_{\leq kr}[E]$. Now, by the Borel-Weyl theorem, we have
\[
\mathrm{H}^0(X,\mathcal{L}_{k\mu,k\nu,kr}) = V_{k\mu}^{\KC*}\otimes V_{k\nu}^{\KC*}\otimes \C_{\leq kr}[E].
\]
This proves the first assertion.

We notice that, for any regular $\mu\in\wedge^*_{\Q,+}$, $r\in\Q_{>0}$, and $k\in\Z_{>0}$ such that $(k\mu,kr)\in\wedge^*\times\Z$, we have $(V_{k\mu}^{\KC*}\otimes V_{k(-w_0\mu)}^{\KC*}\otimes\C_{\leq kr}[E])^{\KC}\neq 0$, since it contains $(V_{k\mu}^{\KC*}\otimes V_{k(-w_0\mu)}^{\KC*}\otimes\C)^{\KC}\neq 0$, where $w_0$ denotes the longest element of the Weyl group of $T$ in $K$. Thus $\mathcal{L}_{\mu,-w_0\mu,0,r}$ is an ample line bundle of $C_{\Q}(\XEC)^{+}$, which proves assertion \emph{(2)}.

Now, by definition of $\Pi_{\Q}(E)$ (Definition \ref{defn:DGIT}) and from assertion \emph{(1)}, it is clear that the image of $\mathrm{pr}(\pi_{E\oplus\C}^{-1}(C_{\Q}(X_{E\oplus\C})^+))$ by the projection map $\mathrm{pr}$ is exactly the set $\Pi_{\Q}(E)$.
\end{proof}

\begin{cor}
\label{cor:PiQ_convexpolyhedralcone}
The set $\Pi_{\Q}(E)$ is a closed convex polyhedral cone of $(\wedge^*_{\Q,+})^2$.
\end{cor}

\begin{proof}
This is induced by Remark \ref{rem:pi_M^-1(CQ(XM))_is_closedpolyhedralcone} and Proposition \ref{prop:SAC_equivalentdef_with_irrep} \emph{(3)}.
\end{proof}

Now, it remains to compute the equations of $\Pi_{\Q}(E)$. Since it is the linear projection of the convex polyhedral cone $C_{\Q}(\XEC)^+$ which we know the equations, see Proposition \ref{prop:equations_of_CQ(XM)}, we may obtain the equations of $\Pi_{\Q}(E)$ from the ones of $C_{\Q}(\XEC)^+$. Unfortunately, here are two issues:
\begin{itemize}
\item Proposition \ref{prop:equations_of_CQ(XM)} gives a set of equations for $C_{\Q}(\XEC)^+$ which is not finite. We will get around this problem in section \ref{section:equations_redundancy}, applying Theorem \ref{thm:coneample_equations_stabilizer} to remove some redundancy in the set of equations of $C_{\Q}(\XEC)^+$. However,  Theorem \ref{thm:coneample_equations_stabilizer} needs $C_{\Q}(\XEC)^+$ to have nonempty interior. This question will be discussed in paragraph \ref{subsection:interior_of_CQ(XEC)^+}.
\item In general, one cannot easily obtain the equations of the projection of some convex polyhedron. But, here, this will be possible using an improvement of Ressayre's well covering pair criterion on $\XEC$, see sections \ref{section:equations_momentpolyhedron_KLambdaE} and \ref{section:wellcoveringpairs}.
\end{itemize}

%

\begin{rem}
\label{rem:nsc_C_Q(X_EoplusC)hasnonemptyinterior}
In the proof of assertion ($2$) of Proposition \ref{prop:SAC_equivalentdef_with_irrep}, we actually proved that $\pi_{E\oplus\C}^{-1}(C_{\Q}(X_{E\oplus\C})^+)$ is not empty, thus $\Pi_{\Q}(E)$ is also not empty. But this can be easily seen by taking $(0,0)$ which is clearly in $\Pi_{\Q}(E)$. Moreover, since $\pi_{E\oplus\C}$ is a linear projection between finite dimensional vector spaces, we obviously have that $C_{\Q}(X_{E\oplus\C})^+$ has nonempty interior in $\PicG(X_{E\oplus\C})_{\Q}$ if and only if $\pi_{E\oplus\C}^{-1}(C_{\Q}(X_{E\oplus\C})^+)$ has nonempty interior in $\wedge^*_{\Q}\times\wedge^*_{\Q}\times\Q$.
\end{rem}


\subsection{Interior of $C_{\Q}(X_{E\oplus\C})^+$}
\label{subsection:interior_of_CQ(XEC)^+}

In this section, we give a necessary and sufficient condition to the nonvacuity of the interior of the set $C_{\Q}(X_{E\oplus\C})^+$. This will be very useful to determine a finite set of equations defining the convex polyhedral cone $C_{\Q}(\XEC)^+$ (cf. section \ref{section:equations_redundancy}).

Let $U$ be the unipotent radical of $B$, and $U^-$ the unipotent radical of the Borel subgroup $B^-$ such that $B\cap B^- = \TC$. We consider the action of $\KC\times \KC$ on $\KC\times E$ defined, for any $(g_1,g_2)\in \KC\times \KC$ and any $(g,v)\in \KC\times E$, by
\[
(g_1,g_2)\cdot (g,v) := (g_1gg_2^{-1},\zeta(g_1)v).
\]
Since $\TC\times \TC$ normalizes the subgroup $U\times U^-$ of $\KC\times \KC$, it induces an action of $\TC\times \TC$ on $\C[\KC\times E]^{U\times U^-}$.

We denote by $\Pi_{\Z}(E)$ the set of integral points of $\Pi_{\Q}(E)$.

\begin{thm}
\label{thm:PiQ_nonemptyinterior}
Let $E$ be a complex representation of $\KC$.
\begin{enumerate}
\item $\Pi_{\Q}(E) = \Q_{>0}\cdot\Pi_{\Z}(E)$.
\item The interior of $\Pi_{\Q}(E)$ is not empty in $(\wedge_{\Q}^*)^2$ if and only if $\ker(\zeta)$ is finite.
\end{enumerate}
\end{thm}

\begin{proof}
The first assertion directly follows from the fact that $\Pi_{\Q}(E)$ is a cone in $(\wedge^*_{\Q})^2$, see Corollary \ref{cor:PiQ_convexpolyhedralcone}.

Our proof of the last assertion is inspired from the proof of \cite[Proposition 3]{lrcone}. By Frobenius' theorem, the $\KC\times \KC$-module $\C[\KC]$ decomposes into
\[
\C[\KC] = \bigoplus_{\nu\in\wedge^*_+}V_{\nu}^{\KC}\otimes V^{\KC*}_{\nu}.
\] 
Then, we have
\begin{align*}
\C[\KC\times E]^{U\times U^-} & = (\C[\KC]\otimes \C[E])^{U\times U^-} = \Big(\bigoplus_{\nu\in\wedge^*_+} V^{\KC}_{\nu}\otimes V_{\nu}^{\KC*}\otimes \C[E]\Big)^{U\times U^-} \\
& = \bigoplus_{\nu\in\wedge^*_+}(V^{\KC}_{\nu}\otimes\C[E])^U\otimes (V_{\nu}^{\KC*})^{U^-},
\end{align*}
where $\TC$ acts on $(V_{\nu}^{\KC*})^{U^-}$ by the character $-\nu$. Hence, the set of weights of $\TC\times \TC$ in $\C[\KC\oplus E]^{U\times U^-}$ is equal to the set of pairs $(\mu,-\nu)\in(\wedge^*_+)^2$ such that $\mu$ is a weight of $\TC$ in $(V^{\KC}_{\nu}\otimes\C[E])^U$, that is, $V^{\KC}_{\mu}\subset V^{\KC}_{\nu}\otimes\C[E]$. And the last assertion is equivalent to $(V_{\mu}^{\KC*}\otimes V^{\KC}_{\nu}\otimes\C[E])^{\KC}\neq 0$, that is $(\mu,-w_0\nu)\in\Pi_{\Z}(E)$.

The map $l:(\mu,\nu)\in(\wedge^*_{\Q})^2\mapsto(\mu,w_0\nu)\in(\wedge^*_{\Q})^2$ is a $\Q$-linear isomorphism, which restricts to a $\Z$-linear map from $(\wedge^*)^2$ to itself. Then, the set of weights of $\TC\times \TC$ in $\C[\KC\oplus E]^{U\times U^-}$ is exactly the set $l(\Pi_{\Z}(E))$ of integral points of $l(\Pi_{\Q}(E))$.

Now, since $\Pi_{\Q}(E) = \Q_{>0}\cdot\Pi_{\Z}(E)$ and $l$ is an isomorphism, the codimension of $\Pi_{\Q}(E)$ in $(\wedge^*_{\Q})^2$ is equal to the dimension of the kernel of the action of $\TC\times \TC$ on the variety $(\KC\times E)/\!/(U\times U^-)$. The actions of $U$ and $U^-$, induced by the action of $\KC\times \KC$ on $\KC\times E$, commute. Thus, we have
\[
\C[\KC\times E]^{U\times U^-} = \left((\C[\KC]\otimes \C[E])^{U^-}\right)^U = \left(\C[\KC]^{U^-}\otimes \C[E]\right)^U,
\]
because $U^-$ acts trivially on $E$, and so on $\C[E]$. But, $BU^-$ is open in $\KC$, then $\C[\KC]^{U^-}$ is canonically isomorphic to a $B$-submodule of $\C[B]$. These two algebras are not isomorphic in general. However, their fields of fractions are isomorphic, since $B$ can be identified to an open subset of $\KC/U$. Hence, $\C(\KC\times E)^{U\times U^-}$ is isomorphic to $\C(B\times E)^U$ as $\TC\times \TC$-modules.

Moreover, $B$ is the semidirect product of the groups $U$ and $\TC$, where $\TC$ normalizes $U$. Hence, there is a canonical $\TC\times \TC$-equivariant isomorphism between $\C(B\times E)^U$ and $\C(\TC\times E)$, where $\TC\times \TC$ acts on $\TC\times E$ by the action induced from the action of $\KC\times \KC$ on $\KC\times E$. Thus, $(\KC\times E)/\!/(U\times U^{-})$ is birationally $\TC\times \TC$-equivariant to the variety $\TC\times E$.

Now, clearly, the kernel of the action of $\TC\times \TC$ on $\TC\times E$ is finite if and only if the kernel of the action of $\TC$ on $E$ is finite, that is, $\ker(\zeta|_{\TC})$ is finite, because $\TC$ is abelian.

It remains to prove that $\ker\zeta$ is finite if and only if $\ker(\zeta|_{\TC})$ is finite. The first implication is obvious.

Suppose $\ker(\zeta|_{\TC})$ is finite, that is $\ker(\zeta)\cap \TC$ is finite. Let $S$ be a maximal torus of $\ker\zeta$, and $\wt{S}$ be a maximal torus of $\KC$ containing $S$. Since $\KC$ is reductive, $\wt{S}$ is conjugate to $\TC$, hence $S$ is conjugate to a subtorus of $\TC$. But $\ker\zeta$ is a normal subgroup of $\KC$, so $S$ is conjugate to a subgroup of $\ker\zeta\cap \TC$, which is finite. Hence $S$ is trivial, and $\ker\zeta$ is finite. Hence, assertion {\it(2)} is proved.
%
\end{proof}

As a consequence of the previous theorem, we have a necessary and sufficient condition on $C_{\Q}(\XEC)^+$ having a nonempty interior. This is the statement of the next corollary. It will be a consequence of the following property of $C_{\Q}(\XEC)^+$.

\begin{lem}
\label{lem:propertyofCQ(XEC)_rcomponent}
For all $(\mu,\nu,r)\in \pi_{E\oplus\C}^{-1}(C_{\Q}(\XEC)^+)$ and all $r'\in\Q$, if $r'\geq r$ then $(\mu,\nu,r')$ is also in $\pi_{E\oplus\C}^{-1}(C_{\Q}(\XEC)^+)$.
\end{lem}

\begin{proof}
Let $(\mu,\nu,r)\in\pi_{E\oplus\C}^{-1}(C_{\Q}(\XEC)^+)$, then there exists a positive integer $n$ such that $(n\mu,n\nu)\in(\wedge^*)^2$, $nr\in\Z$ and $(V_{n\mu}^{\KC*}\otimes V_{n\nu}^{\KC*}\otimes\C_{\leq nr}[E])^{\KC}\neq 0$. It is clear that, for any $m\geq 1$, the $\KC$-module $V_{n\mu}^{\KC*}\otimes V_{n\nu}^{\KC*}\otimes\C_{\leq nr}[E]$ is included in $V_{n\mu}^{\KC*}\otimes V_{n\nu}^{\KC*}\otimes\C_{\leq nm(r+1)}[E]$. Thus, $(V_{n\mu}^{\KC*}\otimes V_{n\nu}^{\KC*}\otimes\C_{\leq nm(r+1)}[E])^{\KC}\neq 0$. That is, for all $m\geq 1$, $(\mu,\nu,m(r+1))$ is in $\pi_{E\oplus\C}^{-1}(C_{\Q}(\XEC)^+)$ too. From the convexity of $\pi_{E\oplus\C}^{-1}(C_{\Q}(\XEC)^+)$ in $\wedge^*_{\Q}\times\wedge^*_{\Q}\times\Q$, since $r+1>0$, we deduce that $(\mu,\nu,r')$ belongs to $\pi_{E\oplus\C}^{-1}(C_{\Q}(\XEC)^+)$ for any rational number $r'\geq r$.
\end{proof}

%

\begin{cor}
\label{cor:coneample_nonemptyinterior}
The convex polyhedral cone $C_{\Q}(\XEC)^+$ has a nonempty interior if and only if $\ker\zeta$ is finite.
\end{cor}

\begin{proof}

By Theorem \ref{thm:PiQ_nonemptyinterior} and Corollary \ref{cor:PiQ_convexpolyhedralcone}, it is sufficient to prove that the convex polyhedral cone $\pi_{E\oplus\C}^{-1}(C_{\Q}(\XEC)^+)$ has a nonempty interior in $\wedge^*_{\Q}\times\wedge^*_{\Q}\times\Q$ if and only if $\Pi_{\Q}(E)$ has a nonempty interior in $(\wedge^*_{\Q})^2$.

But this directly follows from the fact that $\Pi_{\Q}(E)$ is the linear projection of $\pi_{E\oplus\C}^{-1}(C_{\Q}(\XEC)^+)$ in finite dimensional spaces, and from the property stated in Lemma \ref{lem:propertyofCQ(XEC)_rcomponent}.
\end{proof}


\section{Equations redundancy}
\label{section:equations_redundancy}

Here, we reduce the number of equations of $C_{\Q}(\XEC)^+$, given by the well covering pairs, to a finite set of equations determined by pairs $(C,\lambda)$ with one parameter subgroup $\lambda$ satisfying a particular ``admissibility'' property.

%
%

\subsection{Admissible one parameter subgroups}
\label{subsection:admissibility}

We adapt the definition of admissible one parameter subgroup given in \cite[section 7.3.2]{ressayre10}.

\begin{defn}
Let $M$ be a $\TC$-module. A subtorus $S$ of $\TC$ is $M$-\emph{admissible} if it is the neutral component of the stabilizer $(\TC)_v$ of some point $v\in M$.
\end{defn}

\begin{rem}
Equivalently, a subtorus $S$ is $M$-admissible if it is the neutral component of the intersection of the kernels of some characters of $\TC$ on $M$. Using the notations of \cite[7.3.2]{ressayre10}, $S$ is admissible in the sense of Ressayre if and only if $S$ is $(\got{g}\oplus\hat{\got{g}})/\got{g}$-admissible.

In the case of one parameter subgroup $\lambda$ of $\TC$, if we identify $\lambda$ with its generator in the Lie algebra $\got{t}_{\C}$, $\lambda$ is $M$-admissible if and only if $\C\lambda\subset\got{t}_{\C}$ is equal to the intersection in $\got{t}_{\C}$ of the kernels of some weights of $\WT(M)$.
\end{rem}

\begin{rem}
\label{rem:setof_indivisibleMadmissibleops_is_finite}
Since we will only consider finite dimensional representations $M$ of $\KC$, the set of indivisible $M$-admissible one parameter subgroups of $\TC$ will always be finite. Indeed, the set $\WT(M)$ is finite. Hence, we only have a finite number of indivisible one parameter subgroup $\lambda$ such that $\C\lambda$ is equal to the intersection of some weights of $\WT(M)$.
\end{rem}

Now we consider the reductive group $GL_{\C}(E\oplus\C)$. The homomorphism $\zeta:\KC\rightarrow GL_{\C}(E)$ induces an obvious homomorphism of algebraic groups $\zeta\oplus\id_{\C}:\KC\rightarrow GL_{\C}(E\oplus\C)$. Then, the Lie algebra $\got{gl}_{\C}(E\oplus\C)$ of $GL_{\C}(E\oplus\C)$ is a $\KC$-module, hence a $\TC$-module. 

\begin{prop}
\label{prop:admissibility_by_stabilizer}
Let $\lambda$ be a one parameter subgroup of $\TC$, and $C$ a irreducible component of $\XEC^{\lambda}$. If there exists $x\in C$ such that $(\KC)_x^{\circ} = \lambda(\C^*)$, then $\lambda$ is $\got{gl}_{\C}(E\oplus\C)$-admissible.
\end{prop}

\begin{proof}
For simplicity, we denote $\hKC = GL_{\C}(E\oplus\C)$, and let $\hat{B}$ be a Borel subgroup of $\hKC$ containing $(\zeta\oplus\id_{\C})(B)$. Since $\mathbb{P}(E\oplus\C)$ is a homogeneous $\hKC$-space, we identify $X$ with the projective $\KC$-variety $\wt{X} = \KC/B\times \KC/B\times\hKC/\wh{Q}$, using the group homomorphism $\zeta\oplus\id_{\C}$. Here, $\wh{Q}$ is a maximal parabolic subgroup of $\hKC$, stabilizer of some point of $\mathbb{P}(E\oplus\C)$. We can always take $\wh{Q}$ such that $\hat{B}\subset\hat{Q}$. Then the pair $(C,\lambda)$ corresponds to the pair $(\wt{C},\lambda)$ of $\wt{X}$, and $x$ corresponds to some point $\tilde{x}$, whose stabilizer is equal to the stabilizer of $x$. The set $\wt{C}$ is an irreducible component of $\wt{X}^{\lambda}$, so it is the pair $\wt{C}(w,w',\hat{w})$, for some triple $(w,w',\hat{w})\in W \times W \times \wh{W}$, where $\wh{W}$ is the Weyl group of $\hKC$. We write $\tilde{x} = (gwB/B,g'w'B/B,\hat{g}\hat{w}\wh{Q}/\wh{Q})$. We define $\tilde{x}_{\wh{B}} = (gwB/B,g'w'B/B,\hat{g}\hat{w}\wh{B}/\wh{B})\in\wt{C}_{\wh{B}}(w,w',\hat{w})\subset \KC/B\times \KC/B\times\hKC/\wh{B}$. We can easily check that $(\KC)_{\tilde{x}_{\wh{B}}}^{\circ} \subset (\KC)_{\tilde{x}}^{\circ}$, and finally that $(\KC)^{\circ}_{\tilde{x}_{\wh{B}}} = \lambda(\C^*)$. But the irreducible component $\wt{C}_{\wh{B}}(w,w',\hat{w})$ is a complete flag variety of $\KC^{\lambda}\times \KC^{\lambda}\times\hKC^{\lambda}$. Now the result directly follows from \cite[Lemma 17]{ressayre10}, since the $\TC$-module $\bigl(\got{k}_{\C}\oplus\got{gl}_{\C}(E\oplus\C)\bigr)/\got{k}_{\C}$ is canonically isomorphic to the $\TC$-module $\got{gl}_{\C}(E\oplus\C)$.
\end{proof}

The following theorem gives a set of equations for $C_{\Q}(\XEC)^+$. It is similar to \cite[Theorem 9]{ressayre10}.

\begin{thm}
\label{thm:equations_wellcoveringpair_admissible}
Assume that $\ker\zeta$ is finite. Let $(\mu,\nu,r)$ be in $\wedge^*_{\Q,+}\times\wedge^*_{\Q,+}\times\Q_{\geq 0}$. Then $\mathcal{L}_{\mu,\nu,r}\in C_{\Q}(\XEC)^+$ if and only if
\begin{equation}
\label{eq:equation_fromgivenwellcoveringpair}
\langle w\lambda,\mu\rangle + \langle w'\lambda,\nu\rangle + mr \leq 0,
\end{equation}
for all indivisible $\got{gl}_{\C}(E\oplus\C)$-admissible dominant one parameter subgroups $\lambda$ of $\TC$, and for all $(w,w',m)$ in $W/W_{\lambda}\times W/W_{\lambda}\times\Z$ such that $(C(w,w',m),\lambda)$ is a well covering pair (resp. dominant pair) of $X_{E\oplus\C}$.
\end{thm}

\begin{proof}
Let $(C(w,w',m),\lambda)$ be a dominant pair of $\XEC$ such that there exists $x\in C(w,w',m)$ with $(\KC)_x^{\circ} = \lambda(\C^*)$. From Proposition \ref{prop:admissibility_by_stabilizer}, $\lambda$ is $\got{gl}_{\C}(E\oplus\C)$-admissible. 

Applying Corollary \ref{cor:coneample_nonemptyinterior}, we know that $C_{\Q}(\XEC)^{+}$ has a nonempty interior, since $\ker\zeta$ is assumed to be finite. Now the proof immediately results from Theorems \ref{thm:ressayre_equations_amplecone} and \ref{thm:coneample_equations_stabilizer}. The expression of equation \eqref{eq:equation_fromgivenwellcoveringpair} are given by Proposition \ref{prop:equations_of_CQ(XM)}.
\end{proof}

\subsection{Special admissibility for pairs of type $(C(w,w',0),\lambda)$}

We are going to see that we can have a more specific condition of admissibility for one parameter subgroups $\lambda$ appearing in a well covering pair $(C(w,w',0),\lambda)$.

The next proposition is inspired from \cite[Lemma 17]{ressayre10}.

\begin{prop}
\label{prop:subtore_Eadmissible}
Let $S$ be a subtorus of $\TC$. We consider the action of $\KC^S$ on the variety $X' = \KC^S/B^S\times \KC^S/B^S\times\mathbb{P}(E^S\oplus\C)$. If there exists $x\in X'$ such that the neutral component of $(\KC^S)_x$ is equal to $S$, then $S$ is $E$-admissible.
\end{prop}

\begin{proof}
Assume that there exists $x\in X'$ such that $(\KC^S)_x^{\circ} = S$. The torus $S$ acts trivially on $X'$, so for all $y\in X'$, $S\subseteq (\KC^S)_y$. But the dimension of the stabilizer at a point of $X'$ is a lower semicontinuous function, so the set of $x\in X'$ such that $(\KC^S)_x^{\circ} = S$ is open.

Then the general isotropy of $B^S/S$ acting on $\KC^S/B^S\times\mathbb{P}(E^S\oplus\C)$ is finite. Hence, by the Bruhat decomposition, the general isotropy of $B^S/S$ acting on $B^S/\TC\times\mathbb{P}(E^S\oplus\C)$ is finite. The maximal torus $\TC$ normalizes $U^S$ in $B^S$. Thus, the group $\TC/S$ acts on the variety $\left(B^S/\TC\times\mathbb{P}(E^S\oplus\C)\right)/U^S\cong\mathbb{P}(E^S\oplus\C)$ with finite general isotropy. Hence, for general $x\in\mathbb{P}(E^S\oplus\C)$, we have $(\TC)_x^{\circ} = S$.

We notice that $E^S$ is a $\KC^S$-stable open subset of $\mathbb{P}(E^S\oplus\C)$. Consequently, for generic $x\in E^S$, we have $(\TC)_x^{\circ} = S$. Since the general isotropy is equal to the kernel of the action of $\TC$ on $E^S$, $S$ is the neutral component of the kernel of the action of $\TC$ on $E^S$. Thus, $S$ is $E$-admissible.
\end{proof}

\begin{rem}
We can easily check that the set of $\TC$-weights on $\got{gl}_{\C}(E\oplus\C)$ is
\[
\WT(\got{gl}_{\C}(E\oplus\C)) = \left\{\beta-\beta'\ | \ \beta,\beta'\in\WT(E)\right\}\ \cup \ \WT(E).
\]
Thus the set $\WT(E)$ is included in $\WT(\got{gl}_{\C}(E\oplus\C))$. Then a $E$-admissible one parameter subgroup $\lambda$ of $\TC$ is necessarily $\got{gl}_{\C}(E\oplus\C)$-admissible.
\end{rem}

\begin{cor}
\label{cor:misequaltozero_admissible_implies_Eadmissible}
In the statement of Theorem \ref{thm:equations_wellcoveringpair_admissible}, among the well covering (resp. dominant) pairs $(C(w,w',0),\lambda)$ of $X$ such that $\lambda$ is dominant indivisible and $\got{gl}_{\C}(E\oplus\C)$-admissible, we can remove all the pairs with $\lambda$ not $E$-admissible.
\end{cor}

\begin{proof}
Let $(C(w,w',0),\lambda)$ be a dominant pair of $\XEC$ such that there exists $x\in C(w,w',0)$ with $(\KC)_x^{\circ} = \lambda(\C^*)$. It is clear that $C(w,w',0)$ is $\KC^{\lambda}$-equivariantly isomorphic to the variety $X' = \KC^{\lambda}/B^{\lambda}\times \KC^{\lambda}/B^{\lambda}\times\mathbb{P}(E^{\lambda}\oplus\C)$. So from Proposition \ref{prop:subtore_Eadmissible}, $\lambda$ is $E$-admissible. Indeed $\lambda(\C^*)\subset (\KC^{\lambda})_x^{\circ} = (\KC)_x^{\circ}\cap \KC^{\lambda} \subset (\KC)_x^{\circ} = \lambda(\C^*)$, and then $(\KC^{\lambda})_x^{\circ} = \lambda(\C^*)$. And the assertion follows again from Corollary \ref{cor:coneample_nonemptyinterior} and Theorem \ref{thm:coneample_equations_stabilizer}.
\end{proof}


\section{Equations of the moment polyhedron $\Delta_K(K\cdot\Lambda\times E)$}
\label{section:equations_momentpolyhedron_KLambdaE}

We now have all the materials to express the equations of $\Pi_{\Q}(E)$ and, thus, the ones of $\Delta_K(K\cdot\Lambda\times E)$ if $\Phi_E:E\rightarrow\got{k}^*$ is proper, using well covering pairs. We still assume that the Lie group homomorphism $\zeta:\KC\rightarrow GL_{\C}(E)$ has finite kernel. 

\begin{lem}
\label{lem:nc_elementsofPiQ(E)_by_dominantpairs}
Let $(C(w,w',0),\lambda)$ be a dominant pair of $\XEC$. Then, for all $(\mu,\nu)\in\Pi_{\Q}(E)$, we have $\langle w\lambda,\mu\rangle+\langle w'\lambda,\nu\rangle \leq 0$.
\end{lem}

\begin{proof}
Let $(\mu,\nu)\in\Pi_{\Q}(E)$. Then, by Proposition \ref{prop:SAC_equivalentdef_with_irrep} {\it (3)}, there exists $r\in\Q_{\geq0}$ such that $(\mu,\nu,r)\in\pi_{E\oplus\C}^{-1}(C_{\Q}(\XEC)^+)$. Thus, Proposition \ref{prop:equations_of_CQ(XM)} yields that $\langle w\lambda,\mu\rangle+\langle w'\lambda,\nu\rangle+0r\leq 0$, which proves the lemma.
\end{proof}

In fact, we are going to prove that these equations determine completely the polyhedron $\Pi_{\Q}(E)$.

%

Let $\mathscr{P}^{wc}$ be the set of well covering pairs $(C,\lambda)$ of $X_{E\oplus\C}$ such that $\lambda$ is dominant indivisible, and there exists $x\in C$ with $(\KC)_x^{\circ} = \lambda(\C^*)$. This set of well covering pairs is finite, by Proposition \ref{prop:admissibility_by_stabilizer} and Remark \ref{rem:setof_indivisibleMadmissibleops_is_finite}.

Furthermore, for any set $\mathscr{P}$ of dominant pairs of $\XEC$, we define
\[
\mathscr{P}_0 := \{(C(w,w',m),\lambda)\in\mathscr{P} \ | \ m=0\}.
\]


The next statement is the key fact to obtain a set of equations of $\Pi_{\Q}(E)$ from the one of $C_{\Q}(\XEC)^+$.

\begin{thm}
\label{thm:nsc_wellcoveringpair_zeroweight}
Consider the $\KC$-variety $X_M$, where $M$ is a complex representation of $\KC$ such that $0\in\WT(M)$. Let $(w,w',m)\in W/W_{\lambda} \times W/W_{\lambda} \times \Z$ such that $(C(w,w',m),\lambda)$ is a dominant pair  of $X_M$. Then $m$ is non-positive.
\end{thm}

The proof of this theorem uses completely different tools, so we postpone it to section \ref{section:wellcoveringpairs}. In the present case, i.e. $M=E\oplus\C$ with trivial action of $\KC$ on $\C$, the condition $0\in\WT(E\oplus\C)$ is satisfied.


\begin{thm}
\label{thm:equations_PiQ}
Let $\mathscr{P}$ be a set of dominant pairs of $\XEC$ such that $\mathscr{P}^{wc}_0\subseteq\mathscr{P}$. For all $(\mu,\nu)\in\wedge^*_{\Q,+}\times\wedge^*_{\Q,+}$, we have $(\mu,\nu)\in\Pi_{\Q}(E)$ if and only if, for all pairs $(C(w,w',0),\lambda)$ in $\mathscr{P}_0$, we have
\begin{equation}
\label{eq:thm_equations_momentpolyhedron}
\langle w\lambda,\mu\rangle + \langle w'\lambda,\nu\rangle \leq 0.
\end{equation}
\end{thm}

\begin{proof}
It only remains to prove implication ``$\Leftarrow$", since implication ``$\Rightarrow$" is given by Lemma \ref{lem:nc_elementsofPiQ(E)_by_dominantpairs}.

Let $\mathcal{C}$ be the polyhedron defined by the equations \eqref{eq:thm_equations_momentpolyhedron}. Then, we must prove that $\mathcal{C}$ is included in $\Pi_{\Q}(E)$.

Fix $(\mu,\nu)\in\mathcal{C}$. From its definition, $(\mu,\nu)$ satisfies the equations \eqref{eq:equation_fromgivenwellcoveringpair}, with $m = 0$, for all $(C(w,w',0),\lambda)\in\mathscr{P}_0$, and consequently for all $(C(w,w',0),\lambda)\in\mathscr{P}^{wc}_0$. By Theorem \ref{thm:nsc_wellcoveringpair_zeroweight} and the fact that $0\in\WT(E\oplus\C)$, for any pair $(C(w,w',m),\lambda)$ of $\mathscr{P}^{wc}\backslash\mathscr{P}^{wc}_0$, we have $-m> 0$. We define the rational number
\[
r_0 = \max_{C(w,w',m)\in\mathscr{P}^{wc}\backslash\mathscr{P}^{wc}_0}\left\{\frac{\langle w\lambda,\mu\rangle + \langle w'\lambda,\nu\rangle}{-m}\right\}.
\]
This maximum is reached because $\mathscr{P}^{wc}$ is finite. We may assume $r_0>0$, since we can replace $r_0$ by $\max\{r_0,1\}$. Indeed, we still have $\langle w\lambda,\mu\rangle + \langle w'\lambda,\nu\rangle \leq (-m)r_0$, for all $(C(w,w',m),\lambda)\in\mathscr{P}^{wc}\backslash\mathscr{P}^{wc}_0$, because $-m>0$. Hence, $\langle w\lambda,\mu\rangle + \langle w'\lambda,\nu\rangle + mr_0 \leq 0$. Combining this with the definition of $\mathcal{C}$, we deduce that $\langle w\lambda,\mu\rangle + \langle w'\lambda,\nu\rangle + mr_0 \leq 0$ for all $(C(w,w',m),\lambda)\in\mathscr{P}^{wc}$, and Theorem \ref{thm:equations_wellcoveringpair_admissible} shows that $(\mu,\nu,r_0)\in\pi_{E\oplus\C}^{-1}(C_{\Q}(X_{E\oplus\C})^+)$, that is, $(\mu,\nu)\in\Pi_{\Q}(E)$. Finally, we have $\mathcal{C}\subseteq\Pi_{\Q}(E)$.

%
\end{proof}

In particular, we will generally consider the following case.

\begin{cor}
Let $\mathscr{P}_0$ be the set of well covering pairs (resp. covering pairs, resp. dominant pairs) $(C(w,w',0),\lambda)$ such that $\lambda$ is $E$-admissible. Then
\[
\Pi_{\Q}(E) = \{(\mu,\nu)\in(\wedge^*_{\Q,+})^2\ | \ \langle w\lambda,\mu\rangle + \langle w'\lambda,\nu\rangle \leq 0, \forall (C(w,w',0),\lambda)\in\mathscr{P}_0\}.
\]
\end{cor}

Finally, we get different sets of equations describing the moment polyhedra $\Delta_{K\times K}(T^*K\times E)$ and $\Delta_K(K\cdot\Lambda\times E)$.

\begin{cor}
\label{cor:generalequations_DeltaKxK(T*KxE)}
Let $\mathscr{P}$ be a set of dominant pairs of $\XEC$ such that $\mathscr{P}^{wc}_0\subseteq\mathscr{P}$. Assume that the moment map $\Phi_E:E\rightarrow\got{k}^*$ is proper. Then the moment polyhedron $\Delta_{K\times K}(T^*K\times E)$ is the following convex polyhedral cone
\[
\{(\mu,\nu)\in(\got{t}_+^*)^2\ | \ \langle w\lambda,\mu\rangle + \langle w'\lambda,\nu\rangle \leq 0, \forall (C(w,w',0),\lambda)\in\mathscr{P}_0\}.
\]
In particular, for any $\Lambda\in\got{t}_+^*$, we have
\[
\Delta_K(K\cdot\Lambda\times E) = \left\{\xi\in\got{t}_+^*\ |\ \langle w\lambda,\xi\rangle\leq \langle w_0w'\lambda,\Lambda\rangle, \mbox{ for all } (C(w,w',0),\lambda)\in\mathscr{P}_0\right\}.
\]
\end{cor}
%

\begin{proof}
This is a direct consequence of Theorem \ref{thm:equations_PiQ}, Corollary \ref{cor:DeltaClassic_DeltaGIT} and Proposition \ref{prop:relation_between_the_two_momentpolyhedron}.
\end{proof}

Taking $\Lambda=0$ in $\got{t}_+^*$, we deduce the equations of the moment polyhedron $\Delta_K(E):=\Phi_E(E)\cap\got{t}_+^*$ associated to the moment map $\Phi_E:E\rightarrow\got{k}^*$.

\begin{cor}
Assume that $\Phi_E:E\rightarrow\got{k}^*$ is a proper map. Then, we have
\[
\Delta_K(E) = \left\{\xi\in\got{t}^*_+\ | \ \langle w\lambda,\xi\rangle \leq 0, \forall (C(w,w',0),\lambda)\in\mathscr{P}_0^{wc}\right\}.
\]
More generally, let $\Lambda\in\got{t}_+^*$ be central. The set $\Delta_K(K\cdot\Lambda\times E)$
is equal to the polyhedron $\left\{\xi\in\got{t}^*_+\ |\ \langle w\lambda,\xi-\Lambda\rangle \leq 0, \forall (C(w,w',0),\lambda)\in\mathscr{P}_0^{wc}\right\}$. In particular, $\Delta_K(K\cdot\Lambda\times E) = \Lambda + \Delta_K(E)$.
\end{cor}

%
%
%
%
%
%

Among the pairs of $\mathscr{P}_0^{wc}$, the pairs of type $(C(w,w_0w,0),\lambda)$ are special. They give the codimension one faces of the polyhedral cone $\CR(-\WT(E))$. Before proving this result, we need to state the next theorem, which uses results from section \ref{section:wellcoveringpairs}.

\begin{thm}
\label{thm:wellcoveringpairs_C(w,w_0w,m)}
Let $\lambda$ be a dominant one parameter subgroup of $\TC$, and $(w,m)\in W/W_{\lambda}\times\Z$ such that $C(w,w_0w,m)\neq\emptyset$. Then $(C(w,w_0w,m),\lambda)$ is a well covering pair of $\XEC$, and, for all $\beta\in\WT(E)$, we have $\langle\lambda,\beta\rangle\geq m$.
\end{thm}

\begin{proof}
It directly results from Theorem \ref{thm:nsc_wellcoveringpair} and Corollary \ref{cor:info_lengths_elements_for_wellcoveringpair}.
\end{proof}

Next proposition gives partial information about $\Delta_K(K\cdot\Lambda\times E)$ and its faces around $\Lambda$.

\begin{prop}
\label{prop:faces_ConeWtT(E)_and_wellcoveringpairs}
Let $\mathscr{P}^{adm}_0$ denote the set of well covering pairs $(C(w,w',0),\lambda)$ of $\XEC$ such that $\lambda$ is dominant indivisible and $E$-admissible. Then, any pair $(C(w,w_0w,0),\lambda)$ in $\mathscr{P}_0^{adm}$ defines a codimension one face of $\CR(-\WT(E))$. Conversely, any codimension one face of $\CR(-\WT(E))$ arises from some such pair of $\mathscr{P}_0^{adm}$. In particular, for all $\Lambda\in\got{t}_+^*$, we have
\[
\Delta_K(K\cdot\Lambda\times E) \subseteq \Lambda + \CR(-\WT(E)).
\]
\end{prop}

\begin{proof}
Fix $(C(w,w_0w,0),\lambda)\in\mathscr{P}_0^{adm}$. Since $\lambda$ is $E$-admissible, there exists a family $(\beta_{i_1},\ldots,\beta_{i_{n-1}})$ of $n-1$ linearly independent weights in $\got{t}_{\C}^*$, such that $\C\lambda = \cap_{j=1}^{n-1}\ker\beta_{i_j}$. Moreover, $\langle\lambda,\beta\rangle \geq 0$ for all $\beta\in\WT(E)$, by Theorem \ref{thm:wellcoveringpairs_C(w,w_0w,m)}. The Weyl group fixes the set of weights of $\TC$ in $E$, so every weight $\beta\in\WT(E)$ also satisfies $\langle w\lambda,-\beta\rangle \leq 0$. But we also have $\langle w\lambda,-w\beta_{i_k}\rangle = 0$ for all $k=1,\ldots,n-1$. Consequently, the equation $\langle w\lambda,\cdot\rangle \leq 0$ defines a codimension one face of $\CR(-\WT(E))$.

Conversely, let $\mathcal{F}$ be a codimension one face of $\CR(-\WT(E))$. This polyhedral cone is rational, thus there exists an indivisible one parameter subgroup $\lambda_{\mathcal{F}}$ such that $\langle \lambda_{\mathcal{F}},-\beta\rangle \leq 0$ for all $\beta\in\WT(E)$, and $\langle \lambda_{\mathcal{F}},-\beta'_{i_k}\rangle = 0$ for some linearly independent weights $\beta'_{i_1},\ldots,\beta'_{i_{n-1}}$. There exists an element $w\in W$ such that $\lambda = w^{-1}\lambda_{\mathcal{F}}$ is dominant. Moreover, since the vectors $\beta'_{i_k}$, for $k\in\{1,\ldots,n-1\}$, are linearly independent, necessarily we have $\C(w^{-1}\lambda_{\mathcal{F}}) = \cap_{j=1}^{n-1}\ker(w^{-1}\beta'_{i_j})$, and then $w^{-1}\lambda_{\mathcal{F}}$ is a dominant indivisible $E$-admissible one parameter subgroup of $\TC$. By Theorem \ref{thm:wellcoveringpairs_C(w,w_0w,m)}, the pair $(C(w,w_0w,0),w^{-1}\lambda_{\mathcal{F}})$ is well covering, so is in $\mathscr{P}_0^{adm}$. And, evidently, $\langle w(w^{-1}\lambda_{\mathcal{F}}),\cdot\rangle = \langle \lambda_{\mathcal{F}},\cdot\rangle = 0$ is the equation of $\mathcal{F}$.

Moreover, if $x\in\Delta_K(K\cdot\Lambda\times E)$, by Corollary \ref{cor:generalequations_DeltaKxK(T*KxE)}, this element of $\got{t}^*_+$ will satisfy the affine equation
\[
\langle w_{\mathcal{F}}(w_{\mathcal{F}}^{-1}\lambda_{\mathcal{F}}),x-\Lambda\rangle = \langle \lambda_{\mathcal{F}},x-\Lambda\rangle \leq 0,
\]
since the pair $(C(w,w_0w,0),w^{-1}\lambda_{\mathcal{F}})$ is in $\mathscr{P}_0^{adm}$, and here $w'=w_0w$.

This is true for any codimension one face of $\CR(-\WT(E))$, then we have $x -\Lambda \in\CR(-\WT(E))$, and $\Delta_K(K\cdot\Lambda\times E)\subset \Lambda + \CR(-\WT(E))$.
\end{proof}



\section{Well covering pairs of $\KC/B\times \KC/B\times\mathbb{P}(M)$}
\label{section:wellcoveringpairs}

In section \ref{subsection:redgp_git_wellcoveringpairs}, we saw that we have a description of the semiample cone of a smooth projective variety $X$ in terms of linear inequations indexed by well covering pairs of $X$.

In general, it is difficult to determine the set of all well covering pairs of some projective variety. However, we study a particular variety with a very interesting form. In \cite{ressayre10}, Ressayre gives a necessary and sufficient condition on $(C,\lambda)$ to be a well covering pair, in the case of a variety of the form $Y=\KC/Q\times \hKC/\wh{Q}$, where $\KC$ is a connected reductive subgroup of a connected complex reductive group $\hKC$, and $Q$ (resp. $\wh{Q}$) a parabolic subgroup of $\KC$ (resp. $\hKC$). Our variety $X_M=\KC/B\times \KC/B\times \mathbb{P}(M)$ is an example of such variety, when $M$ is a finite dimensional $\KC$-module.

This section is dedicated to the computation of the set of well covering pairs of the projective variety $X_M$.

\subsection{The main criterion}
\label{subsection:main_criterion}

We use the same notations from subsection \ref{subsection:notations_preliminaries_GIT}, for groups $\TC\subset B\subset \KC$. Let $\zeta:\KC\rightarrow GL(M)$ be a complex representation of $\KC$, with $r:=\dim_{\C}M$. We denote by $\got{t}_{\C}$, $\got{b}$ and $\got{k}_{\C}$ the respective Lie algebras. Fix a dominant one parameter subgroup $\lambda$ of $\TC$. Denote by $P(\lambda)$ the parabolic subgroup of $\KC$ corresponding to $\lambda$. It is standard since $\lambda$ is dominant.

Let $P$ be any parabolic subgroup of $\KC$ containing $B$, and $W_P$ the Weyl group of its Levi subgroup. For any $w\in W/W_P$, the set $\overline{BwP/P}$ is called the Schubert variety corresponding to $w$. The fundamental classes of the Schubert varieties form a basis of the free $\Z$-module $\homol{\KC/P}$, and we define $\{\sigma_{w}^{P}\ | \ w\in W/W_P\}$ to be the dual basis in $\coh{\KC/P}$. We denote by $\pt := \sigma_{w_0}^{P}$ the class of the point, that is, the generator of $\mathrm{H}^{2\dim_{\C}(\KC/P)}(\KC/P,\Z)$.

From now on, we consider $P=P(\lambda)$. Let $W_{\lambda}=W_{P(\lambda)}$ be the Weyl group of $\KC^{\lambda}$, and let $W^{\lambda}$ denote the set of maximal length representatives of $W/W_{\lambda}$. We define the map $\jmath : gB\in \KC/B \mapsto gP(\lambda)\in \KC/P(\lambda)$, and the induced map
\[
\jmath^*: \coh{\KC/P(\lambda)}\longrightarrow\coh{\KC/B}.
\]
Since $\jmath$ is surjective, the ring homomorphism $\jmath^*$ is injective. It is a well-known fact that, for all $w$ in $W^{\lambda}$, we have $\jmath^*(\sigma_{w_0w}^{P(\lambda)}) = \sigma_{w_0w}^B$ in $\coh{\KC/B}$, since $w_0w$ is the shortest element of $w_0wW_{\lambda}$. In particular, we have $\jmath^*(\pt) = \sigma_{w_0w_{\lambda}}^B$, where $w_{\lambda}$ is the longest element in $W_{\lambda}$. See \cite{brion_flagvar} for more details.

We recall that for all $k\in\Z$, the subspace $M_{\lambda,k}$ is defined by $M_{\lambda,k} := \{v\in M\ | \ \lambda(t)\cdot v = t^k v, \forall t\in\C^*\}$. For any $m\in\Z$, we define the subspaces
\[
M_{< m} := \bigoplus_{k <m} M_{\lambda,k}, \qquad \text{and} \qquad M_{\geq m} := \bigoplus_{k\geq m} M_{\lambda,k}.
\]
We notice that the set $\WT(M_{<m})$ is equal to the set of weights $\beta$ of $M$ such that $\langle\lambda,\beta\rangle<m$. For all $\beta\in\WT(M)$, we denote by $n_{\beta}$ the multiplicity of the weight $\beta$ on $M$, that is, $n_{\beta} := \dim_{\C}M_{\beta}$, where $M_{\beta}$ is the weight space of $M$ with weight $\beta$.

Let $\Theta : \wedge^*\rightarrow\mathrm{H}^2(\KC/B,\Z)$ be the morphism that sends a weight $\mu\in\wedge^*$ of $\TC$, onto the first Chern class $\Theta(\mu) = c_1(\mathcal{L}_{\mu})$, of the line bundle $\mathcal{L}_{\mu}$ with weight $\mu$.

We denote by $\rho$ the half sum of the positive roots of $\got{k}_{\C}$.

\renewcommand{\theenumi}{\textit{\roman{enumi}}}

\begin{thm}
\label{thm:nsc_wellcoveringpair}
Let $(w,w',m)\in W^{\lambda}\times W^{\lambda}\times \Z$ such that $C(w,w',m)$ is nonempty. The pair $(C(w,w',m),\lambda)$ of $X_M$ is well covering if and only if, either $w' = w_0ww_{\lambda}$ and $M_{<m}=0$, or the following assertions are satisfied,
\begin{enumerate}
\item $\sigma_{w_0w}^B\,.\,\sigma_{w_0w'}^B\,.\,\prod_{\beta\in\WT(M_{<m})}\Theta(-\beta)^{n_{\beta}} = \jmath^*(\pt)$,
\item $\langle w\lambda+w'\lambda,\rho\rangle+\sum_{k<m}(m-k)\dim_{\C}(M_{\lambda,k}) = 0$.
\end{enumerate}
\end{thm}

We postpone the proof of Theorem \ref{thm:nsc_wellcoveringpair} to subsection \ref{subsection:proof_thm_wellcoveringpair}. In the above statement, the case $w' = w_0ww_{\lambda}$ and $M_{<m}=0$ is a specific case of ($i$) and ($ii$), if we use the convention that a product of an empty set of elements is equal to $\sigma_{\id}^B= 1$.

We finish this subsection with a corollary of Theorem \ref{thm:nsc_wellcoveringpair}, which gives a simple necessary condition with some pair $(C(w,w',m),\lambda)$ being well covering, by checking the lengths of $w$ and $w'$.

Let $(C(w,w',m),\lambda)$ be a well covering pair. By Theorem \ref{thm:nsc_wellcoveringpair}, we must have $\sigma_{w_0w}^B\,.\,\sigma_{w_0w'}^B\,.\,\prod_{\beta\in\WT(M_{<m})}\Theta(-\beta)^{n_{\beta}} = \jmath^*(\pt)$, so we have a necessary condition in terms of degree. Indeed, this equation implies that
\[
l(w_0w) + l(w_0w') + \dim_{\C}(M_{<m}) = \dim_{\C}(\KC/P(\lambda)).
\]
But, for all $u\in W$, we have $l(w_0u) = l(w_0)-l(u)$, because $w_0$ is the longest element in $W$. Thus, the previous equation is equivalent to
\[
l(w) + l(w') = \dim_{\C}(M_{<m}) + 2l(w_0) - \dim_{\C}(\KC/P(\lambda)),
\]
that we can also write: $l(w) + l(w') = l(w_0) + l(w_{\lambda}) + \dim_{\C}(M_{<m})$.

\begin{cor}
\label{cor:info_lengths_elements_for_wellcoveringpair}
Let $(w,w',m)\in W^{\lambda} \times W^{\lambda} \times \Z$ such that $(C(w,w',m),\lambda)$ is a well covering pair. Then
\begin{equation}
\label{eq:equality_lengths}
l(w) + l(w') = l(w_0) + l(w_{\lambda}) + \dim_{\C}(M_{<m}).
\end{equation}
In particular, we have $M_{<m} = 0$ if and only if $w' = w_0ww_{\lambda}$.
\end{cor}

\begin{proof}
It remains to prove the last assertion. From Remark \ref{rem:N_and_M<m} in subsection \ref{subsection:parametrization} and the proof of Theorem \ref{thm:nsc_wellcoveringpair}, if $M_{<m} = 0$, then $w' = w_0ww_{\lambda}$, and a simple verification yields
\begin{equation}
\label{eq:equality_lengths_casew'=w_0w}
l(w) + l(w') = l(w_0) + l(w_{\lambda}).
\end{equation}
Here, we use the fact that $l(w'w_{\lambda}) = l(w') - l(w_{\lambda})$, since $w'$ is the longest element of $w'W_{\lambda}$.

Conversely, if $w'=w_0ww_{\lambda}$, then, equations \eqref{eq:equality_lengths} and \eqref{eq:equality_lengths_casew'=w_0w} are satisfied, thus
$\dim_{\C}(M_{<m})=0$, and finally $M_{<m}=0$.
\end{proof}

\begin{rem}
The formula \ref{eq:equality_lengths} is also true for any covering pair $(C(w,w',m),\lambda)$ of $X_M$, because the condition $\sigma_{w_0w}^B\,.\,\sigma_{w_0w'}^B\,.\,\prod_{\beta\in\WT(M_{<m})}\Theta(-\beta)^{n_{\beta}} = \jmath^*(\pt)$ is valid for any such pair (actually, it is a necessary and sufficient condition, see Proposition \ref{prop:cohomological_criterion_dominant/covering_pairs}).
\end{rem}

\subsection{Notations and parametrization}
\label{subsection:parametrization}

In order to prove Theorem \ref{thm:nsc_wellcoveringpair}, we have to choose a good identification of $\mathbb{P}(M)$ with some flag variety $\hKC/\wh{Q}$, defining a morphism $f:\KC\rightarrow\hKC$ and a maximal torus $\hTC$ such that $f(\TC) \subset\hTC$ and $\hTC\subset\wh{Q}$. The choice of such identification leads us to parametrize our vector space $M$.

Recall that we have fixed a dominant one parameter subgroup $\lambda$ of $\TC$. We denote by $\hat{P}(\lambda)$ the parabolic subgroup of $GL_{\C}(M)$ associated to the one parameter subgroup $\zeta\circ\lambda$. Then, clearly we have the inclusions $\zeta(B)\subseteq\zeta(P(\lambda))\subseteq\hat{P}(\lambda)$. Thus, there exists a Borel subgroup $\hat{B}$ of $\hat{P}(\lambda)$ containing $\zeta(B)$. The group $\hat{B}$ is also a Borel subgroup of $GL_{\C}(M)$. Let $\hat{T}_{\C}$ a maximal torus of $\hat{B}$ containing $\zeta(T_{\C})$.

Since $\zeta(\lambda(\C^*))\subseteq\hat{T}_{\C}$, for each $k\in\Z$, the subspace $M_{\lambda,k}$ is a $\hat{T}_{\C}$-submodule. Now, list the weights of $\lambda$ on $M$ in decreasing order $\{k_1 > k_2 > \cdots > k_s \} = \{k\in\Z \, | \, M_{\lambda,k}\neq 0\}$, and define the flag $0\subsetneq V_1 \subseteq \cdots \subseteq V_s = M$, where $V_i = \bigoplus_{j=1}^i M_{\lambda,k_j}$. One can see that this flag is $\hat{P}(\lambda)$-stable, hence $\hat{B}$-stable. But $\hat{B}$ is a Borel group, so there exists a $\hat{B}$-stable complete flag $0\subsetneq V'_1 \subseteq \cdots \subseteq V'_r = M$ such that $V'_{\dim V_1+\cdots+\dim V_i} = V_i$, for all $i=1,\ldots,s$. Now, take a basis $\mathcal{B}_{\lambda}=(u_1,\ldots,u_r)$ whose elements are common eigenvectors of the action of $\hTC$ on $M$ such that
\[
\mathrm{Vect}(u_1,\ldots,u_i) = V'_i, \quad \forall i=1,\ldots,r.
\]
For all $i\in\{1,\ldots,r\}$, let $\beta_i$ be the weight of $\WT(M)$ such that $u_i\in M_{\beta_i}$. From the definition of the complete flag $0\subsetneq V'_1 \subseteq \cdots \subseteq V'_r = M$, we must have $\langle\lambda,\beta_i\rangle \geq \langle\lambda,\beta_{i+1}\rangle$, for all $i\in\{1,\ldots,r-1\}$.

With this choice of basis of $M$, we can identify $\mathbb{P}(M)$ with the flag variety $\hKC/\wh{Q}$, where $\hKC=GL_r(\C)$ and $\wh{Q}$ is the stabilizer in $\hKC$ of the line $\C u_1$, that is, the maximal parabolic subgroup
\begin{equation}
\label{eq:parabolicsubgroup_wh{Q}}
\wh{Q} = \left(\begin{array}{c|ccc}
* & * & \ldots & * \\\hline
0 & * & \ldots & * \\
\vdots & \vdots & \ddots & \vdots \\
0 & * & \ldots & *
\end{array}\right).
\end{equation}
This identification is canonically defined by the map $\hat{g}\wh{Q}\in \hKC/\wh{Q}\mapsto[\hat{g}\cdot u_1]\in\mathbb{P}(M)$. This map is $\KC$-equivariant if we equip $\hKC/\wh{Q}$ with the action induced by the group homomorphism
\[
\begin{array}{rccl}
f_{\lambda} : & \KC & \longrightarrow & \hKC=GL_r(\C) \\
& g & \longmapsto & \Mat_{\mathcal{B}_{\lambda}}\left(\rho(g)\right).
\end{array}
\]

Clearly, $\hTC$ is identified to the maximal torus of the diagonal matrices of $\hKC$, and $\wh{B}$ the Borel subgroup of the upper triangular matrices of $\hKC$. By definition of $\hat{Q}$, we have $\hTC\subset\wh{B}\subset\wh{Q}$

From now on, we identify the $\KC$-variety $X_M$ with the product of flag varieties $\KC/B\times \KC/B\times\hKC/\wh{Q}$, endowed with the induced action of $\KC$. Following this identification, we have another simple and complete description of irreducible components of the projective variety $X_M^{\lambda}$. Indeed, we have
\[
X_M^{\lambda} = \bigcup_{\stackrel{w,w'\in W/W_{\lambda}}{\hat{w}\in\wh{W}_{\wh{Q}}\backslash \wh{W}/\wh{W}_{\lambda}}} \KC^{\lambda}w^{-1}B/B\times \KC^{\lambda}w'^{-1}B/B\times \hKC^{\lambda}\hat{w}^{-1}\wh{Q}/\wh{Q},
\]
where $\wh{W}$ (resp. $\wh{W}_{\lambda}$, resp. $\wh{W}_{\wh{Q}}$) is the Weyl group of $\hKC$ (resp. of the Levi $\hKC^{\lambda}$ of $\wh{P}(\lambda)$, resp. of the Levi of $\wh{Q}$).

For $i=1,\ldots,r-1$, let $s_i$ denote the standard simple permutation endomorphisms associated to the canonical basis of $\C^r$. That is, $s_i(u_i)=u_{i+1}$, $s_i(u_{i+1})=u_i$, and $s_i(u_k)=u_k$ if $k\notin\{i,i+1\}$. We define the elements $\hat{w}_k = s_1\circ\ldots\circ s_{k-1}$, for $k=2,\ldots r$, and $\hat{w}_1 = \id$, of $\wh{W}$.

\begin{lem}
\label{lem:C_m_and_wQwN}
Let $N=\dim_{\C}(M_{\geq m})$. Then $\mathbb{P}(M_{\lambda,m}) = \hKC^{\lambda}\hat{w}_N^{-1}\wh{Q}/\wh{Q}$, and $\hat{w}_{\wh{Q}}\hat{w}_N$ is the longest element of the class $\wh{W}_{\wh{Q}}\hat{w}_N\wh{W}_{\lambda}$ in $\wh{W}_{\wh{Q}}\backslash\wh{W}/\wh{W}_{\lambda}$.
\end{lem}

\begin{proof}
We know that $\mathbb{P}(M_{\lambda,m})$ is an irreducible component of $\mathbb{P}(M)^{\lambda}$. Hence, $\mathbb{P}(M_{\lambda,m}) = \hKC^{\lambda}\hat{w}^{-1}\wh{Q}/\wh{Q}$, for some $\hat{w}\in\wh{W}_{\wh{Q}}\backslash\wh{W}/\wh{W}_{\lambda}$. Moreover, by the identification $\hat{g}\wh{Q}\in\hKC/\wh{Q}\mapsto[\hat{g}\cdot u_1]\in\mathbb{P}(M)$, we have $\mathbb{P}(M_{\lambda,m}) = \hKC^{\lambda}\hat{w}^{-1}\wh{Q}/\wh{Q}$ if and only if $\hat{w}^{-1}\cdot u_1\in M_{\lambda,m}$ (clearly, this does not depend on the representative of the class $\wh{W}_{\wh{Q}}\hat{w}\wh{W}_{\lambda}$). Hence $\mathbb{P}(M_{\lambda,m}) = \hKC^{\lambda}\hat{w}_N^{-1}\wh{Q}/\wh{Q}$, since $\hat{w}_N^{-1}\cdot u_1 = s_{N-1}\circ\ldots\circ s_1(u_1) = u_N$ is in $M_{\lambda,m}$, by the definitions of $N$ and $\mathcal{B}_{\lambda}$, and the chosen ordering.

The class $\wh{W}_{\wh{Q}}\hat{w}_N\wh{W}_{\lambda}$ is stable by left multiplication by an element of $\wh{W}_{\wh{Q}}$. It is then a disjoint union of classes of $\wh{W}_{\wh{Q}}\backslash\wh{W}$. By Lemma \ref{lem:shorter_elements_of_W_mod_WQ_right}, the set $\{\hat{w}_k\ | \ k=1,\ldots,r\}$ is a system of shorter representatives of the classes of $\wh{W}_{\wh{Q}}\backslash\wh{W}$. 
Thus, $\wh{W}_{\wh{Q}}\hat{w}_N\wh{W}_{\lambda}$ decomposes into a disjoint union of classes defined by left multiplication by $\wh{W}_{\wh{Q}}$, that is,
\[
\wh{W}_{\wh{Q}}\hat{w}\wh{W}_{\lambda} = \bigcup_{k\in\{1,\ldots,r\}, \langle\lambda,\beta_k\rangle=m}\wh{W}_{\wh{Q}}\hat{w}_k.
\]
By its definition, $N$ is the largest integer of the set $\{k\in\{1,\ldots,r\}\ | \langle\lambda,\beta_k\rangle = m\}$. But, for two integers $k<k'$, we have $l(v\hat{w}_k) = l(v) + k < l(v) + k' = l(v\hat{w}_{k'})$, for all $v\in\wh{W}_{\wh{Q}}$, by Lemma \ref{lem:shorter_elements_of_W_mod_WQ_right}. We conclude that $\hat{w}_{\wh{Q}}\hat{w}_N$ is the longest element of $\wh{W}_{\wh{Q}}\hat{w}_N\wh{W}_{\lambda}$.
\end{proof}

\begin{rem}
\label{rem:N_and_M<m}
We notice that $N = \dim_{\C}(M_{\geq m}) = r-\dim_{\C}(M_{<m})$. Thus, $N=r$ if and only if $M_{<m}=0$.
\end{rem}

\medskip
Thus, the irreducible component $C(w,w',m)$ of $X_M^{\lambda}$ is identified with the irreducible component $C(w,w',\hat{w}_{\wh{Q}}\hat{w}_N)$ of $(\KC/B\times \KC/B\times\hKC/\wh{Q})^{\lambda}$.

Let $i_{\lambda}: g\in \KC \mapsto (g,f_{\lambda}(g))\in \KC\times \hKC$ be the injection we are going to study. The induced map $i_{\lambda} : \KC/P(\lambda) \rightarrow \KC/P(\lambda)\times \hKC/\wh{P}(\lambda)$ is a closed immersion. This map induces the map in cohomology $i_{\lambda}^* : \coh{\KC/P(\lambda)\times \hKC/\wh{P}(\lambda)} \rightarrow \coh{\KC/P(\lambda)}$. We also define the map
\[
f_{\lambda}^{P(\lambda)} : gP(\lambda)\in \KC/P(\lambda)\mapsto f_{\lambda}(g)\wh{P}(\lambda)\in\hKC/\wh{P}(\lambda),
\]
and the induced map $(f_{\lambda}^{P(\lambda)})^*:\coh{\hKC/\wh{P}(\lambda)} \rightarrow \coh{G/P(\lambda)}$. We define similarly the maps $f_{\lambda}^B$ and $(f_{\lambda}^B)^*$.

\subsection{Cohomological criterion}
\label{subsection:cohomological_criterion}

Let $\tKC$ be a complex connected reductive group, and $\KC$ a connected reductive subgroup. Let $i:\KC\hookrightarrow\tKC$ denote the embedding of groups. Let us fix a maximal torus $\TC$ (resp. $\tTC$) and a Borel subgroup $B$ (resp. $\wt{B}$) of $\KC$ (resp. $\tKC$) such that $\TC\subseteq B\subseteq\wt{B}\supseteq\tTC\subseteq\TC$. Let $Q$ (resp. $\wt{Q}$) be a parabolic subgroup of $\KC$ (resp. $\tKC$) containing $\TC$ (resp. $\tTC$). We recall that $\rho$ (resp. $\hat{\rho}$, resp. $\tilde{\rho}$) denotes the half sum of positive roots of $\got{k}_{\C}$ (resp. $\hat{\got{k}}_{\C}$, resp. $\tilde{\got{k}}_{\C}$).

\begin{lem}[\cite{ressayre10}, Lemma 14]
\label{lem:nsc_pair_is_covering/dominant_generalcase}
Let $\lambda$ be a dominant one parameter subgroup of $\TC$. Let $(w,\tilde{w})\in W\times\wt{W}$ be such that $w$ (resp. $\tilde{w}$) is the longest element in the class $W_{Q}wW_{\lambda}$ (resp. $\wt{W}_{\wt{Q}}\tilde{w}\wt{W}_{\lambda}$). Then: 
\begin{enumerate}
\item the pair $(C(w,\tilde{w}),\lambda)$ is dominant if and only if $\sigma_{w_0w}^{P(\lambda)}\,.\,i^*\bigl(\sigma_{\tilde{w}_0\tilde{w}}^{\wt{P}(\lambda)}\bigr) \neq 0$,
\item the pair $(C(w,\tilde{w}),\lambda)$ is covering if and only if $\sigma_{w_0w}^{P(\lambda)}\,.\,i^*\bigl(\sigma_{\tilde{w}_0\tilde{w}}^{\wt{P}(\lambda)}\bigr) = \pt$.
\end{enumerate}
\end{lem}

In the case of the variety $X_M$, we are able to improve the above formulas.

Fix once and for all a dominant one parameter subgroup $\lambda$  of $\TC$. We apply Lemma \ref{lem:nsc_pair_is_covering/dominant_generalcase} to $\tKC=\KC\times \hKC$, and the injection $i=i_{\lambda}:\KC\hookrightarrow \KC\times\hKC$ defined in subsection \ref{subsection:parametrization}.

Let $(w,w',m)\in W^{\lambda} \times W^{\lambda} \times \Z$ be such that $C(w,w',m)$ is not empty. We define $\tilde{w} = (w',\hat{w}_{\wh{Q}}\hat{w}_N)\in\wt{W}=W\times\wh{W}$. The pair $(w,\tilde{w})$ satisfies the assumption of Lemma \ref{lem:nsc_pair_is_covering/dominant_generalcase}. And, by Lemma \ref{lem:C_m_and_wQwN}, $(C(w,w',m),\lambda)$ is a dominant (resp. covering, resp. well covering) pair if and only if $(C(w,\tilde{w}),\lambda)$ is.

\begin{lem}
\label{lem:cohomological_alternatives}
Let $(w,w',m)\in W^\lambda\times W^\lambda\times\Z$. Then:
\begin{enumerate}
\item either $M_{<m}=0$, and then $\jmath^*\left(\sigma_{w_0w}^{P(\lambda)}\,.\,i_{\lambda}^*\bigl(\sigma_{(w_0w',\hat{w}_0\hat{w}_{\hat{Q}}\hat{w}_N)}^{P(\lambda)\times\hat{P}(\lambda)}\bigr)\right) = \sigma_{w_0w}^{B}\,.\,\sigma_{w_0w'}^{B}$,
\item or $\displaystyle\jmath^*\left(\sigma_{w_0w}^{P(\lambda)}\,.\,i_{\lambda}^*\bigl(\sigma_{(w_0w',\hat{w}_0\hat{w}_{\hat{Q}}\hat{w}_N)}^{P(\lambda)\times\hat{P}(\lambda)}\bigr)\right) = \sigma_{w_0w}^{B}\,.\,\sigma_{w_0w'}^{B}\,.\,\displaystyle\prod_{\beta\in\WT(M_{< m})}\Theta(-\beta)^{n_{\beta}}$.
\end{enumerate}
\end{lem}

%

\begin{proof}
First, we notice that we have
\[
i^*\left(\sigma_{\tilde{w}_0\tilde{w}}^{\wt{P}(\lambda)}\right) = i_{\lambda}^*\left(\sigma_{w_0w'}^{P(\lambda)}\otimes\sigma_{\hat{w}_0\hat{w}_{\wh{Q}}\hat{w}_N}^{\wh{P}(\lambda)}\right) = \sigma_{w_0w'}^{P(\lambda)}\,.\,\left(f_{\lambda}^{P(\lambda)}\right)^*\left(\sigma_{\hat{w}_0\hat{w}_{\wh{Q}}\hat{w}_N}^{\wh{P}(\lambda)}\right),
\]
since $i_{\lambda}$ is the composition of maps $(\id\times f_{\lambda}^{P(\lambda)})\circ\Delta$, where $\Delta$ is the diagonal map $\KC/P\rightarrow \KC/P\times \KC/P$. Thus, $i_{\lambda}^* = \Delta^*\circ(\id\times (f_{\lambda}^{P(\lambda)})^*)$, and $\Delta^*$ is the cup product.

Since $\jmath^*$ is a ring homomorphism for the cup product, and $w_0w$ (resp. $w_0w'$) is the shortest element of $w_0wW_{\lambda}$ (resp. $w_0wW_{\lambda}$), we have
\[
\jmath^*\left(\sigma_{w_0w}^{P(\lambda)}\,.\,\sigma_{w_0w'}^{P(\lambda)}\,.\,(f_{\lambda}^{P(\lambda)})^*\bigl(\sigma_{\hat{w}_0\hat{w}_{\hat{Q}}\hat{w}_N}^{\hat{P}(\lambda)}\bigr)\right) = \sigma_{w_0w}^{B}\,.\,\sigma_{w_0w'}^{B}\,.\,\jmath^*\left((f_{\lambda}^{P(\lambda)})^*\bigl(\sigma_{\hat{w}_0\hat{w}_{\hat{Q}}\hat{w}_N}^{\hat{P}(\lambda)}\bigr)\right).
\]
Moreover, the following commutative diagram
\[
\begin{CD}
   \KC/B @>\jmath>> \KC/P(\lambda) \\
   @Vf_{\lambda}^BVV @VVf_{\lambda}^{P(\lambda)}V\\
   \hKC/\wh{B} @>\hat{\jmath}>> \hKC/\wh{P}(\lambda)
\end{CD}
\]
yields the equalities
\[
\jmath^*\left((f_{\lambda}^{P(\lambda)})^*\bigl(\sigma_{\hat{w}_0\hat{w}_{\wh{Q}}\hat{w}_N}^{\wh{P}(\lambda)}\bigr)\right) = (f_{\lambda}^B)^*\left(\hat{\jmath}^*\bigl(\sigma_{\hat{w}_0\hat{w}_{\wh{Q}}\hat{w}_N}^{\wh{P}}\bigr)\right) = (f_{\lambda}^B)^*\bigl(\sigma_{\hat{w}_0\hat{w}_{\wh{Q}}\hat{w}_N}^{\wh{B}}\bigr),
\]
since $\hat{w}_0\hat{w}_{\wh{Q}}\hat{w}_N$ is the shortest element of its class in $\wh{W}/\wh{W}_{\lambda}$. Thus,
\[
\jmath^*\left(\sigma_{w_0w}^{P(\lambda)}\,.\,\sigma_{w_0w'}^{P(\lambda)}\,.\,(f_{\lambda}^{P(\lambda)})^*\bigl(\sigma_{\hat{w}_0\hat{w}_{\hat{Q}}\hat{w}_N}^{\hat{P}(\lambda)}\bigr)\right) = \sigma_{w_0w}^B\,.\,\sigma_{w_0w'}^B\,.\,(f_{\lambda}^B)^*\bigl(\sigma_{\hat{w}_0\hat{w}_{\wh{Q}}\hat{w}_N}^{\wh{B}}\bigr).
\]

We use the Chevalley formula to compute $(f_{\lambda}^B)^*\bigl(\sigma_{\hat{w}_0\hat{w}_{\wh{Q}}\hat{w}_N}^{\wh{B}}\bigr)$. All the necessary results are gathered in paragraph \ref{subsection:Chevalleyformula}. From Lemma \ref{lem:relation_between_wcheck_what}, we have $\hat{w}_0\hat{w}_{\wh{Q}}\hat{w}_{N} = s_{r-1}\ldots s_{N+1}s_N$, when $N\in\{1,\ldots,r-1\}$, and $\hat{w}_0\hat{w}_{\wh{Q}}\hat{w}_{r} = \id$. Thus, equation \eqref{eq:sigmawcheck_image_by_flambdacheck_final} yields
\[
(f_{\lambda}^B)^*(\sigma_{\hat{w}_0\hat{w}_{\wh{Q}}\hat{w}_N}^{\wh{B}}) = (f_{\lambda}^B)^*(\sigma_{s_{r-1}\ldots s_N}^{\wh{B}}) = \Theta\bigl((-\beta_{N+1}) \ldots (-\beta_{r})\bigr),
\]
when $N<r$, and
\[
(f_{\lambda}^B)^*(\sigma_{\hat{w}_0\hat{w}_{\wh{Q}}\hat{w}_r}^{\wh{B}}) = (f_{\lambda}^B)^*(\sigma_{\id}^{\wh{B}}) = \sigma_{\id}^B.
\]
Furthermore, we notice that $\Theta(\beta_{N+1}\ldots\beta_{r}) = \prod_{\beta\in\WT(M_{<m})}\Theta(\beta)^{n_{\beta}}$, since the elements $\beta_{N+1},\ldots,\beta_r$ are the weights of $M_{<m}$ with multiplicity.

We thus have the following alternative:
\begin{itemize}
\item either $N=r$, and then $\jmath^*\left(\sigma_{w_0w}^{P(\lambda)}\,.\,i_{\lambda}^*\bigl(\sigma_{(w_0w',\hat{w}_0\hat{w}_{\hat{Q}}\hat{w}_N)}^{P(\lambda)\times\hat{P}(\lambda)}\bigr)\right) = \sigma_{w_0w}^{B}\,.\,\sigma_{w_0w'}^{B}$,
\item or $1\leqslant N<r$, and then
\[
\jmath^*\left(\sigma_{w_0w}^{P(\lambda)}\,.\,i_{\lambda}^*\bigl(\sigma_{(w_0w',\hat{w}_0\hat{w}_{\hat{Q}}\hat{w}_N)}^{P(\lambda)\times\hat{P}(\lambda)}\bigr)\right) = \sigma_{w_0w}^{B}\,.\,\sigma_{w_0w'}^{B}\,.\,\prod_{\beta\in\WT(M_{< m})}\Theta(-\beta)^{n_{\beta}}.
\]
\end{itemize}
But $N=r-\dim_{\C}M_{<0}$ from Remark \ref{rem:N_and_M<m}. Hence we can deduce the statement of this lemma.
\end{proof}

The next proposition gives a cohomological criterion which determines the dominant (resp. covering) pairs of $X_M$.

\begin{prop}
\label{prop:cohomological_criterion_dominant/covering_pairs}
Let $(w,w',m)\in W^{\lambda}\times W^{\lambda}\times\Z$ such that 
 $C(w,w',m)$ is non-empty. Then, the pair $(C(w,w',m),\lambda)$ is dominant (resp. covering) if and only if
\begin{enumerate}
\item either $M_{<m} = 0$ and $\sigma_{w_0w}^{B}\,.\,\sigma_{w_0w'}^{B}\neq 0$ (resp. $M_{<m} = 0$ and $w' = w_0ww_{\lambda}$),
\item or $\sigma_{w_0w}^{B}\,.\,\sigma_{w_0w'}^{B}\,.\,\prod_{\beta\in\WT(M_{< m})}\Theta(-\beta)^{n_{\beta}} \neq 0$ (resp. $\dots=\jmath^*(\pt)$).
\end{enumerate}
\end{prop}

This proposition results from Lemmas \ref{lem:nsc_pair_is_covering/dominant_generalcase} and \ref{lem:cohomological_alternatives}, from the injectivity of the morphism $\jmath^*:\coh{\KC/P(\lambda)}\rightarrow\coh{\KC/B}$, and from Lemma \ref{lem:product_twoschubertclasses_inG/P_isequalto_ptclass} which is given just below. This lemma is a well-known fact, generalizing results for Borel subgroup case proved for example in \cite[Lemma 1 and Proposition 1]{demazure}, or \cite{chevalley}.

\begin{lem}
\label{lem:product_twoschubertclasses_inG/P_isequalto_ptclass}
Let $(w,w')\in (W^{\lambda})^2$ be such that $l(w)+l(w')\leq l(w_0)+l(w_{\lambda})$, then 
\[
\sigma_{w}^{P(\lambda)}\,.\,\sigma_{w'}^{P(\lambda)} = \left\{\begin{array}{ll}
\pt & \text{if $w'=w_0ww_{\lambda}$,} \\
0 & \text{otherwise}.
\end{array}\right.
\]
\end{lem}

The criterion given by Proposition \ref{prop:cohomological_criterion_dominant/covering_pairs} induces a very interesting property satisfied by dominant pairs. Indeed, when the maximal torus $\TC$ has a nonrivial fixed element in $M$, or, equivalently, if the zero weight is a weight of the representation $\zeta$, we are able to say a little bit more about the integer $m$ appearing in dominant pairs. This is the statement of Theorem \ref{thm:nsc_wellcoveringpair_zeroweight}. 
%

\begin{proof}[Proof of Theorem \ref{thm:nsc_wellcoveringpair_zeroweight}]
Let $(w,w',m)\in W^{\lambda} \times W^{\lambda} \times \Z$ such that $C(w,w',m)$ is not empty in $(X_M)^{\lambda}$. Assume $m>0$. Hence, the zero weight is in $\WT(M_{<m})$, and $\prod_{\beta\in\WT(M_{<m})}\Theta(-\beta)^{n_{\beta}} = 0$, because $\Theta$ is a morphism of $\Z$-modules. Finally, Proposition \ref{prop:cohomological_criterion_dominant/covering_pairs} implies that the pair $(C(w,w',m),\lambda)$ is not dominant.
\end{proof}

\subsection{Proof of Theorem \ref{thm:nsc_wellcoveringpair}}
\label{subsection:proof_thm_wellcoveringpair}

We are going to prove the necessary and sufficient condition for a pair of $X_M$ to be well covering. Next theorem from \cite{ressayre10} gives a criterion in a more general context. Here we use the notations introduced in section \ref{subsection:cohomological_criterion}.

\renewcommand{\theenumi}{\textit{\alph{enumi}}}

\begin{thm}[\cite{ressayre10}, Proposition 11]
\label{thm:nsc_pair_is_wellcovering_generalcase}
Let $\lambda$ be a dominant one parameter subgroup of $\TC$. Let $(w,\tilde{w})\in W\times\wt{W}$ be such that $w$ (resp. $\tilde{w}$) is the longest element in the class $W_{Q}wW_{\lambda}$ (resp. $\wt{W}_{\wt{Q}}\tilde{w}\wt{W}_{\lambda}$). Then the pair $(C(w,\tilde{w}),\lambda)$ is well covering if and only if the two following assertions are satisfied :
\begin{enumerate}
\item $\sigma_{w_0w}^{P(\lambda)}\,.\,i^*\bigl(\sigma_{\tilde{w}_0\tilde{w}}^{\wt{P}(\lambda)}\bigr) = \pt$,

\item $\langle \lambda,\rho+w^{-1}\rho\rangle + \langle i^*(\lambda),\tilde{\rho}+\tilde{w}^{-1}\tilde{\rho}\rangle = \langle\lambda,2\rho\rangle$.
\end{enumerate}
\end{thm}

From Lemmas \ref{lem:nsc_pair_is_covering/dominant_generalcase} and \ref{lem:cohomological_alternatives}, Proposition \ref{prop:cohomological_criterion_dominant/covering_pairs} and Remark \ref{rem:N_and_M<m}, we already know that assertion {\it (a)} of Theorem \ref{thm:nsc_pair_is_wellcovering_generalcase} is equivalent to the following alternative:
\begin{itemize}
\item either $M_{<m} = 0$ and $w' = w_0ww_{\lambda}$,
\item or $M_{<m} \neq 0$ and $\sigma_{w_0w}^{B}\,.\,\sigma_{w_0w'}^{B}\,.\,\prod_{\beta\in\WT(M_{< m})}\Theta(-\beta)^{n_{\beta}} = \jmath^*(\pt)$)
\end{itemize}

Hence, it comes down to prove that assertion ($b$) of Theorem \ref{thm:nsc_pair_is_wellcovering_generalcase} for $(w,\tilde{w})$ with $\tilde{w}=(w',\hat{w}_{\wh{Q}}\hat{w}_N)$, is equivalent to the linear equation 
\[
\langle w\lambda+w'\lambda,\rho\rangle+\sum_{k<m}(m-k)\dim_{\C}(M_{\lambda,k}) = 0.
\]
We will essentially use the next lemma.

\begin{lem}
\label{lem:gammahat}
We have $\hat{\rho}+(\hat{w}_{\wh{Q}}\hat{w}_N)^{-1}\hat{\rho} = \sum_{l=N+1}^{r}\halpha_{N,l}$.
\end{lem}

\begin{proof}
By Lemma \ref{lem:relation_between_wcheck_what}, we know that $w_0\hat{w}_{\wh{Q}}\hat{w}_N = \check{w}_N := s_{r-1}\ldots s_{N+1}s_N$. Thus, $\hat{w}_{\wh{Q}}\hat{w}_N = w_0\check{w}_N$, and
\[
(\hat{w}_{\wh{Q}}\hat{w}_N)^{-1}\hat{\rho} + \hat{\rho} = \check{w}_N^{-1}\hat{w}_0\hat{\rho} + \hat{\rho}.
\]
But it is clear that $(\check{w}_N^{-1}\hat{w}_0\hat{\rho} + \hat{\rho})$ is the sum of the roots $\halpha\in\wh{\got{R}}^+$ such that $\hat{w}_0\check{w}_N(\halpha)$ is positive. Since $\hat{w}_0$ switches the sets $\wh{\got{R}}^+$ and $\wh{\got{R}}^-$, $(\check{w}_N^{-1}\hat{w}_0\hat{\rho} + \hat{\rho})$ is the sum of the positive roots $\halpha$ such that $\check{w}_N(\halpha)$ is negative. These roots are the roots $\halpha_{N,l}$, for $N+1 \leq l \leq r$. Then the result follows.
\end{proof}

Clearly, we have
\[
\langle i_{\lambda}^*(\lambda),\tilde{\rho}+\tilde{w}^{-1}\tilde{\rho}\rangle = \langle\lambda,\rho+w'^{-1}\rho\rangle + \langle f_{\lambda}^*(\lambda),\hat{\rho}+(\hat{w}_{\wh{Q}}\hat{w}_N)^{-1}\hat{\rho}\rangle.
\]
By Lemma \ref{lem:gammahat}, using the fact that, for all $k=N+1,\ldots,r$, $\langle\lambda,\hat{\alpha}_{N,k}\rangle = \langle\lambda,\beta_N\rangle-\langle\lambda,\beta_k\rangle$, we obtain
\[
\langle f_{\lambda}^*(\lambda),\hat{\rho}+(\hat{w}_{\wh{Q}}\hat{w}_N)^{-1}\hat{\rho}\rangle = \sum_{l=N+1}^r(\langle\lambda,\beta_N\rangle -\langle\lambda,\beta_l\rangle) = \sum_{k<m}(m-k)\dim_{\C}(M_{\lambda,k}).
\]
Now, assertion ($b$) of Theorem \ref{thm:nsc_pair_is_wellcovering_generalcase} may be written as follows,
\begin{align*}
\langle \lambda,\rho+w^{-1}\rho\rangle + \langle i^*(\lambda),\tilde{\rho}+\tilde{w}^{-1}\tilde{\rho}\rangle-\langle\lambda,2\rho\rangle & = \langle \lambda,w^{-1}\rho+w'^{-1}\rho\rangle \\
& + \sum_{k<m}(m-k)\dim_{\C}(M_{\lambda,k}).
\end{align*}
Hence, we have proved that assertion ($b$) of Theorem \ref{thm:nsc_pair_is_wellcovering_generalcase} for $(w,\tilde{w})$, is equivalent to $\langle w\lambda+w'\lambda,\rho\rangle+\sum_{k<m}(m-k)\dim_{\C}(M_{\lambda,k}) = 0$.

\begin{lem}
\label{lem:equation_lambda_extremalcase}
For any $w\in W^{\lambda}$, we have $\langle w\lambda+w_0ww_{\lambda}\lambda,\rho\rangle=0$.
\end{lem}

\begin{proof}
This directly follows from $w_0\rho = -\rho$, and $w_{\lambda}\lambda = \lambda$, since $w_{\lambda}\in W_{\lambda}$.
\end{proof}


In the case of $N=r$, we have $M_{< m} = 0$, from Remark \ref{rem:N_and_M<m}. And this clearly yields that $\sum_{k<m}(m-k)\dim_{\C}(M_{\lambda,k}) = 0$. Lemma \ref{lem:equation_lambda_extremalcase} implies that the equation
\[
\langle w\lambda+w'\lambda,\rho\rangle+\sum_{k<m}(m-k)\dim_{\C}(M_{\lambda,k}) = 0
\]
is always satisfied. That is, if $w'=w_0ww_{\lambda}$, then the assertion ($b$) is satisfied.

\bigskip

Finally, we proved that the pair $(C(w,w',m),\lambda)$ is well covering if and only if we have either $w'=w_0ww_{\lambda}$ and $M_{<m}=0$ (that is, $N=r$), or the following two assertions are satisfied:
\begin{enumerate}
\item $\sigma_{w_0w}^B\,.\,\sigma_{w_0w'}^B\,.\,\prod_{\beta\in\WT(M_{<m})}\Theta(-\beta)^{n_{\beta}} = \sigma_{w_0w_{\lambda}}^B$,
\item $\langle w\lambda+w'\lambda,\rho\rangle+\sum_{k<m}(m-k)\dim_{\C}(M_{\lambda,k}) = 0$.
\end{enumerate}
This is exactly the statement of Theorem \ref{thm:nsc_wellcoveringpair}. This ends the proof of Theorem \ref{thm:nsc_wellcoveringpair}.


\section{The moment polyhedron $\Delta_K(G\cdot\Lambda)$}
\label{section:the_moment_polyhedra_DeltaK(GLambda)}

Let $G$ be a simple, connected, noncompact, real Lie group with finite center, $K$ a maximal compact subgroup of $G$, and $\got{g} = \got{k}\oplus\got{p}$ the associated Cartan decomposition. In this section, we finally give a set of equations of the Kirwan polyhedron associated with the projection on $\got{k}^*$ of an holomorphic coadjoint orbit of $G$.

\subsection{Equations of the moment polyhedron $\Delta_K(G\cdot\Lambda)$}
\label{subsection:equations_momentpolyhedron_G0Lambda}

We use notations from subsection \ref{subsection:Kirwanpolyhedron_KLambdaE}. Recall that we can define a $K$-invariant Hermitian structure on $\got{p}^-$,
\[
h_{\got{p}^-}(X,Y) = B_{\got{g}}(X, Y) - iB_{\got{g}}(X, \ad(z_0)Y), \mbox{ for all } X,Y\in\got{p},
\]
using the canonical $K$-equivariant isomorphism $\got{p}\rightarrow\got{p}^-$. The associated $K$-invariant symplectic form is
\[
\Omega_{\got{p}^-}(X,Y) = B_{\got{g}}(X, \ad(z_0)Y), \mbox{ for all } X,Y\in\got{p}.
\]
We get a group homomorphism $\zeta: K\rightarrow U(\got{p}^-)$.

Let $\Lambda\in\Chol$. We know from Corollary \ref{cor:equality_of_momentpolyhedra} that $\Delta_K(G\cdot\Lambda) = \Delta_K(K\cdot\Lambda\times\got{p}^-)$. We thus apply the results of section \ref{section:equations_momentpolyhedron_KLambdaE}, for $E = \got{p}^-$. As we saw in subsection \ref{subsection:Kirwanpolyhedron_KLambdaE}, the corresponding moment map $\Phi_{\got{p}^-}$ is proper.

\begin{lem}
The kernel of $\zeta:K\rightarrow U(\got{p}^-)$ is equal to the center $Z(G)$ of $G$. In particular, $\ker\zeta$ is finite.
\end{lem}

\begin{proof}
Fix $k\in\ker\zeta$. For all $Y\in\got{p}^-$, $\Ad(k)Y = Y$. Since $\got{p}$ and $\got{p}^-$ are $K$-equivariantly isomorphic for the adjoint action, $\Ad(k)Y' = Y'$ for all $Y'\in\got{p}$. Thus, for all $Y,Z\in\got{p}$, we have $\Ad(k)[Y,Z] = [\Ad(k)Y,\Ad(k)Z] = [Y,Z]$. Consequently, $\Ad(k)$ is the identity map on $[\got{p},\got{p}] = \got{k}$, because $\got{g}$ is simple (see \cite{knapp}, problem $24$ page $430$). By linearity, since $\got{g} = \got{k} \oplus \got{p}$, $\Ad(k)$ fixes every element in $\got{g}$.

Now, if $g\in G$, the Cartan decomposition on the Lie group $G$ yields $g = h\exp(Y)$, with $h\in K$ and $Y\in\got{p}$. Since $K$ is compact and connected, the map $\exp:\got{k}\rightarrow K$ is surjective. Hence, there exists $X\in\got{k}$ such that $h = \exp(X)$. This implies $kgk^{-1} = \exp(\Ad(k)X)\exp(\Ad(k)Y) = \exp(X)\exp(Y) = g$, and $k$ commutes with $g$. Thus $\ker\zeta\subseteq Z(G)$. And clearly $Z(G)\subseteq\ker\zeta$. This proves the lemma.
\end{proof}

\begin{thm}[Equations of $\Delta_K(G\cdot\Lambda)$]
\label{thm:equations_DeltaK_G0Lambda}
Let $\mathscr{P}$ be a set of dominant pairs of $X_{\got{p}^-\oplus\C}$ such that $\mathscr{P}^{wc}_0\subseteq\mathscr{P}$. Fix $\Lambda\in\Chol$ and $\xi\in\got{t}^*_+$. Then $\xi$ is in $\Delta_K(G\cdot\Lambda)$ if and only if $\xi$ verifies the equations $\langle w\lambda,\xi\rangle\leq\langle w_0w'\lambda,\Lambda\rangle$ for all pairs $(C(w,w',0),\lambda)\in\mathscr{P}_0$.
\end{thm}

\begin{proof}
The proof directly results from Corollary \ref{cor:generalequations_DeltaKxK(T*KxE)}.
\end{proof}

We notice that the weights of $\TC$ on $\got{p}^-$ are the noncompact negative roots, i.e. $\WT(\got{p}^-) = \got{R}_n^-$. Moreover, a dominant one parameter subgroup $\lambda$ is $\got{p}^-$-admissible if there exists $n-1$ noncompact positive roots $\beta_1,\ldots,\beta_{n-1}$ such that $\C\lambda = \cap_{i=1}^{n-1}\ker\beta_i$.

\begin{prop}
Fix $\Lambda\in\Chol$. Then $\Delta_K(K\cdot\Lambda\times\got{p}^-)\subset\Chol$. In particular, for all $\Lambda\in\Chol$, we have $\Delta_K(G\cdot\Lambda)\subset\Chol$.
\end{prop}

\begin{proof}
By Proposition \ref{prop:faces_ConeWtT(E)_and_wellcoveringpairs}, we have $\Delta_K(K\cdot\Lambda\times\got{p}^-) \subset (\Lambda + \CR(\got{R}_n^+))\cap\got{t}^*_+$. We denote by $\bmin$ the smaller noncompact positive root. Then $\Chol = \{\xi\in\got{t}^*_+\ | \ (\bmin,\xi) > 0\}$, where $(\cdot,\cdot)$ is the inner product on $\got{t}^*$ induced by the Killing form on $\got{g}$. For any root $\beta\in\got{R}_n^+$,  $\bmin+\beta$ is not a root, nor $0$, so by \cite[Proposition 2.48 (e)]{knapp}, $(\bmin,\beta)\geq 0$. So any element $\xi$ in $\CR(\got{R}_n^+)$ verifies $(\bmin,\xi)\geq 0$. And if $\Lambda$ is in $\Chol$, then we have $(\bmin,\xi)> 0$ for all $\xi\in(\Lambda + \CR(\got{R}_n^+))$. Now, it is clear that $\Delta_K(K\cdot\Lambda\times\got{p}^-) \subset (\Lambda + \CR(\got{R}_n^+))\cap\got{t}^*_+ \subset \Chol$, as soon as $\Lambda$ is in $\Chol$.
\end{proof}

\begin{rem}
By Theorem \ref{thm:momentpolyhedron_and_irreduciblerepresentations}, the moment polyhedron $\Delta_K(K\cdot\Lambda\times\got{p})$, and thus $\Delta_K(G\cdot\Lambda)$, is closely related to the action of $K$ on $\got{p}^-$ and $\C[\got{p}^-]$. In fact, $\got{p}^-$ is an irreducible complex representation of $K$. But we have a more stronger property: the algebra $\C[\got{p}^-]$ is multiplicity-free as $K$-module. More specifically, we have,
\[
\C[\got{p}^-]=\sum_{p_1\geq\ldots\geq p_r\geq 0}V^K_{p_1\gamma_1+\ldots+p_r\gamma_r},
\]
where $\{\gamma_1,\ldots,\gamma_r\}$ is the maximal set constructed inductively: $\gamma_1$ is the maximal positive noncompact root, and, for $i=1,\ldots,r-1$, $\gamma_{i+1}$ is the maximal noncompact positive root strongly orthogonal to $\gamma_1,\ldots,\gamma_i$. Here, two roots $\alpha$ and $\beta$ are \emph{strongly orthogonal} if neither of $\alpha\pm\beta$ is a root. See \cite{johnson80} for more details.

In the classical cases, we have the following list, from \cite{knapp},

\begin{enumerate}
\item when $G=Sp(\R^{2n})$, $K=U(n)$ and $\got{p}^-\cong \mathrm{Sym}^2(\C^n)$, with standard action of $U(n)$;
\item when $G=SO^*(2n)$, $K=U(n)$ and $\got{p}^-\cong \wedge^2\C^n$, with standard action of $U(n)$;
\item when $G=SU(p,q)$, $K=S(U(p)\times U(q))$ and $\got{p}^-\cong M_{p,q}(\C)$, with $U(p)$ (resp. $U(q)$) acting by left (resp. right) multiplication on the space of matrices $M_{p,q}(\C)$;
\item when $G=SO(p,2)$, $K=SO(p)\times SO(2)$, and $\got{p}^-\cong \C^p$, with standard action of $SO(p)$ and $SO(2)=S^1$ on $\C^p$.
\end{enumerate}

In the next paragraph, we will compute effectively the equations of $\Delta_K(K\cdot\Lambda\times\got{p})$ without taking account of this property of $\C[\got{p}^-]$. But it would be interesting to see if it brings other geometric properties or simplification in the computation of the equations.
\end{rem}

\subsection{Examples of moment polyhedra}
\label{subsection:examples_of_moment_polyhedra}

In this subsection, we give a complete description of the moment polyhedron $\Delta_K(G\cdot\Lambda)$ when $G = Sp(\R^{2n})$, $SU(n,1)$ and $SU(2,2)$. In the first two cases, the maximal compact subgroup $K$ is isomorphic to $U(n)$. Let $T$ denote the maximal torus of diagonal matrices of $U(n)$. Let $\got{t}$ be the Lie algebra of $T$. We have a canonical basis $(e_1^*,\ldots,e_n^*)$ in $\got{t}^*$, where $e_j^*(i\diag(h_1,\ldots,h_n)) = h_j$. 
Let $(e_1,\ldots,e_n)$ be its dual basis in $\got{t}$. The compact roots of $\got{g}$ are the $\alpha_{i,j} = e_i^*-e_j^*$, with $1\leq i,j \leq n$, $i\neq j$. The set of compact positive roots is $\got{R}_c^+ = \{\alpha_{i,j}\ | \ 1\leq i<j\leq n\}$.

\subsubsection{Moment polyhedron of $Sp(\R^{2n})$, $n\geq 2$}

When $G = Sp(\R^{2n})$, the noncompact roots are the linear forms $\pm\beta_{i,j} = \pm(e_i^*+e_j^*)$, for $1\leq i\leq j\leq n$. The set of noncompact positive roots is $\got{R}^+_n=\{\beta_{i,j}\ | \ 1\leq i\leq j\leq n\}$. The smallest noncompact negative root is $-\beta_{1,1}=\bmin$. Any noncompact negative root $\beta$ is the sum
\[
\beta = \bmin + \sum_{\alpha\in\got{R}_c^+} n_{\alpha}\alpha,
\]
with $n_{\alpha}\in\Z_{\geq 0}$ for all $\alpha\in\got{R}_c^+$. Then for all dominant one parameter subgroups $\lambda$, $\langle\lambda,\bmin\rangle \leq \langle\lambda,\beta\rangle$, for all roots $\beta\in\got{R}_n^-$.

The Chevalley Formula yields $\Theta(-\bmin) = \Theta(\beta_{1,1}) = 2\sigma^B_{s_{\alpha_{1,2}}}$. Thus, for all $k\geq 1$, 
\[
\sigma_{w_0w}^B\,.\,\sigma_{w_0w'}^B\,.\,\prod_{\beta\in\WT(\got{p}^-_{< m})}\Theta(-\beta)^{n_{\beta}}= 2p\sigma_{w_0w_{\lambda}}^B+\ldots,
\]
for some $p\in\Z$
. We deduce that the only well covering pairs are of the form $(C(w,w_0ww_{\lambda},0),\lambda)$, with $w\in W^{\lambda}$, such that $\got{p}^-_{<0}= 0$.

The unique indivisible dominant $\got{p}^-$-admissible one parameter subgroup with $\got{p}^-_{<0}=0$ is $\lambda = (0,\ldots,0,-1)$. Theorem \ref{thm:equations_DeltaK_G0Lambda} implies the following complete description of the moment polyhedron of the Hamiltonian $K$-manifold $G\cdot\Lambda$.

\begin{thm}
For $G = Sp(\R^{2n})$ and $\Lambda\in\Chol$, we have
\begin{align*}
\Delta_K(G\cdot\Lambda) & = (\Lambda + \CR(\got{R}_n^+))\cap \got{t}^*_+ \\
& = \{(\xi_1,\ldots,\xi_n)\in\got{t}^*_+ \ | \ \xi_i\geq\Lambda_i, 
\mbox{ for all } i=1,\ldots,n\}.
\end{align*}
\end{thm}

\begin{figure}[htb]
\label{figure:polyèdremoment_sp(r,n)}
\begin{center}
\setlength{\unitlength}{1cm}
\resizebox{8cm}{8cm}{
\begin{pspicture}(10,10)
\definecolor{Couleur1}{rgb}{0,0.5,0.7}
\definecolor{Couleur2}{rgb}{0.6,0.4,0.2}

\pspolygon[linecolor=lightgray,fillstyle=solid,linewidth=0pt,fillcolor=lightgray](2,0)(2,10)(3.5,10)(4.5,9.9)(5,9.8)(6,9.6)(7,9.2)(8,8.6)(8.8,7.8)(9.3,7)(9.7,6.3)(10,5)(9.7,3.7)(9.3,3)(8.8,2.2)(8,1.4)(7,0.8)(6,0.4)(5,0.2)(4.5,0.1)(3.5,0)
\pspolygon[linecolor=gray,fillstyle=solid,linewidth=0pt,fillcolor=gray](2,1)(2,9.8)(3.5,9.8)(4.5,9.7)(5.5,9.4)(6.5,9)(7.6,8.4)(8.5,7.5)
\pspolygon[linecolor=lightgray,fillstyle=solid,linewidth=0pt,fillcolor=lightgray](3.5,5)(2,6.5)(2,9.7)(3.5,9.7)(4.5,9.6)(5.5,9.3)(6.3,8.95)(7,8.5)

\psdot(3.5,5)

\psline[linewidth=1pt](2,0)(2,10)
\psline[linewidth=1pt](1,0)(8.8,7.8)
\psline[linewidth=1pt](3.5,5)(2,6.5)
\psline[linewidth=1pt](3.5,5)(7,8.5)

\psline[linewidth=1pt]{->}(2,1)(2,2)
\psline[linewidth=1pt]{->}(2,1)(3,1)
\psline[linewidth=1pt]{->}(2,1)(3,2)
\psline[linewidth=1pt]{->}(2,1)(1,2)

\rput(3.5,4.5){$\Lambda$}
\rput(3,0.7){$\alpha$}
\rput(3,1.5){$\beta_{1,1}$}
\rput(2.4,2){$\beta_{1,2}$}
\rput(1,1.5){$\beta_{2,2}$}
\rput(8.1,4){$\got{t}^*_+$}
\rput(6.3,6.8){$\Chol$}
\rput(3.7,7.5){$\Delta_K(G\cdot\Lambda)$}

\end{pspicture}
}
\end{center}
\caption{Polyhedron $\Delta_K(Sp(\R^4)\cdot\Lambda)$ for $\Lambda\in\Chol$}
\end{figure}

\subsubsection{Moment polyhedron of $SU(n,1)$, $n\geq 2$}
\label{subsubsection:momentpolyhedron_su(n,1)}

When $G = SU(n,1)$, the noncompact roots are the linear form $\pm\beta_k = \pm(e_k^* + \sum_{j=1}^n e_j^*)$ for all $k=1,..,n$, with $\got{R}_n^+ =\{\beta_k\ | \ k=1,\ldots,n\}$ and $\got{R}_n^- =\{-\beta_k\ | \ k=1,\ldots,n\}$. The smallest noncompact negative root is $\bmin = -\beta_1$, and for all $k=2,\ldots,n$, we have
\[
-\beta_k = \bmin + \alpha_{1,2} + \ldots + \alpha_{k-1,k}.
\]
For all $k=1,\ldots,n$, we define the one parameter subgroup $\lambda_k = (n+1)e_k - \sum_{j=1}^ne_j$. We can easily check that $\C\lambda_k = \cap_{j\neq k} \ker(-\beta_j)$. Thus, $\pm\lambda_k$ is $\got{p}^-$-admissible, and the set of $\got{p}^-$-admissible indivisible one parameter subgroup is $\{\pm\lambda_k\ | \ k=1,\ldots,n\}$. Furthermore, the dominant indivisible $\got{p}^-$-admissible one parameter subgroups are $\lambda_1$ et $-\lambda_n$.
The latter, $-\lambda_n$, yields $\got{p}^-_{<0} = 0$, because $\langle\lambda,\bmin\rangle = 0$. Then $-\lambda_n$ gives the equations of the convex cone $\Lambda + \CR(\got{R}_n^+)$. These equations are $\langle-\lambda_k,\xi-\Lambda\rangle \leq 0$, for all $k$ in $\{1,\ldots,n\}$, that is, $\langle\lambda_k,\xi\rangle \geq \langle\lambda_k,\Lambda\rangle$.

Unlike $Sp(\R^{2n})$, we will see there are other equations, those of $\lambda_1$. We have $\langle\lambda_1,\bmin\rangle = -(n+1)$, and $\langle\lambda_1,-\beta_k\rangle = 0$, for all $k=2,\ldots,n$. Furthermore, $\langle\lambda_1,\alpha_{1,2}\rangle=n+1 > 0$ and $\langle\lambda_1,\alpha_{k,k+1}\rangle = 0$, for all $k\in\{2,\ldots,n-1\}$. Hence, $P(\lambda_1)$ is equal to the parabolic subgroup \eqref{eq:parabolicsubgroup_wh{Q}} for $r=n$.
The well covering pairs with $\lambda_1$ and $m=0$ are the pairs $(C(w,w',0),\lambda_1)$, with $(w,w')\in W^{\lambda}\times W^{\lambda}$, such that
\[
\sigma_{w_0w}^B\,.\,\sigma_{w_0w'}^B\,.\,\Theta(-\bmin) = \sigma_{w_0w_{\lambda_1}}^B,
\]
and
\[
\langle w\lambda_1+w'\lambda_1,\rho\rangle-\langle\lambda_1,\bmin\rangle = 0,
\]
by Theorem \ref{thm:nsc_wellcoveringpair}. Moreover, using Corollary \ref{cor:info_lengths_elements_for_wellcoveringpair} and Lemma \ref{lem:shorter_elements_of_W/WQ}, we can show that the well covering pairs must be of the form $(C(\hat{w}_k^{-1},\hat{w}_{n-k+2}^{-1},0),\lambda_1)$ with $k\in\{2,\ldots,n\}$.

Lemma \ref{lem:shorter_elements_of_W/WQ} shows $w_0\hat{w}_k^{-1}w_{\lambda_1} = \hat{w}_{n-k+1}^{-1} = s_{n-k}\ldots s_1$, for all $k=1,\ldots,n$ (for $k=n$, $w_0\hat{w}_n^{-1}w_{\lambda_1} = \id$), and $w_0w_{\lambda_1} = \hat{w}_{n}^{-1}=s_{n-1}\ldots s_1$. We can compute $\Theta(-\bmin) = \sigma_{s_1}^B$ by the Chevalley formula. We are reduced to compute $\sigma_{s_1}^B\,.\,\sigma_{\hat{w}_{n-k+1}^{-1}}^B\,.\,\sigma_{\hat{w}_{k-1}^{-1}}^B$ for all $k=1,\ldots,n$.

\begin{lem}
\label{lem:formula_powerof_sigma_s_1}
For all $k=1,\ldots,n-1$, we have $(\sigma_{s_1}^B)^k = \sigma_{s_k\ldots s_1}^B = \sigma_{\hat{w}_{k+1}^{-1}}^B$. Moreover, $(\sigma_{s_1}^B)^n = 0$.
\end{lem}

\begin{proof}
This proof is similar to that of Lemmas \ref{lem:cupproduct_sk_wcheckk} and \ref{lem:cupproduct_skminus1_wcheckk}. We use Chevalley's formula combined to Lemma \ref{lem:wcheck-1_plus1_jequalsk+2} in order to show that $\sigma_{s_1}^B\,.\,\sigma_{s_k\ldots s_1}^B = \sigma_{s_{k+1}\ldots s_1}^B$ for all $k\in\{1,\ldots, n-2\}$. Then an obvious induction on $k$ proves the first assertion.
\end{proof}

\begin{lem}
\label{lem:dualitypairing_lambda1_gamma_SU(n,1)}
For all $k=1,\ldots,n$, we have $\langle\lambda_1,\hat{w}_k\rho+\hat{w}_{n-k+2}\rho\rangle = \langle\lambda_1,\bmin\rangle$.
\end{lem}

\begin{proof}
We use the fact that $\hat{w}_{n-k+2} = w_{\lambda_1}\hat{w}_{k-1}w_0$. This yields
\begin{align*}
\langle\lambda_1,\hat{w}_k\rho+\hat{w}_{n-k+2}\rho\rangle & = \langle\lambda_1,\hat{w}_k\rho+\hat{w}_{k-1}w_0\rho\rangle = \langle\hat{w}_{k-1}\lambda_1,s_k\rho-\rho\rangle \\
& = -\langle\hat{w}_{k-1}\lambda_1,\alpha_{k,k+1}\rangle = -\langle\lambda_1,\alpha_{1,k+1}\rangle.
\end{align*}
Thus, $\langle\lambda_1,\hat{w}_k\rho+\hat{w}_{n-k+2}\rho\rangle = -(n+1) = \langle\lambda_1,\bmin\rangle$.
\end{proof}

We deduce from the above lemmas that the pair $(C(\hat{w}_k^{-1},\hat{w}_{n-k+2}^{-1},0),\lambda_1)$ is well covering, for all $k=2,\ldots,n$. The pair $(C(\hat{w}_k^{-1},\hat{w}_{n-k+2}^{-1},0),\lambda_1)$ brings the equation
\begin{equation}
\label{eq:equation_pairswithhatwkminus1}
\langle \hat{w}_k^{-1}\lambda_1,x\rangle\leq\langle w_0\hat{w}_{n-k+2}^{-1}\lambda_1,\Lambda\rangle.
\end{equation}
By definition, $\hat{w}_{k}^{-1} = s_{k-1}\ldots s_1$ for all $2\leq k\leq n$. Hence, $\hat{w}_{k}^{-1}\lambda_1 = s_{k-1}\ldots s_1\lambda_1 = \lambda_{k}$. Notice that we have
\[
\langle w_0\hat{w}_{n-k+2}^{-1}\lambda_1,\Lambda\rangle = \langle w_0\hat{w}_{n-k+2}^{-1}w_{\lambda_1}\lambda_1,\Lambda\rangle = \langle \hat{w}_{k-1}^{-1}\lambda_1,\Lambda\rangle = \langle\lambda_{k-1},\Lambda\rangle.
\]
because $w_{\lambda_1}$ is in $W_{\lambda_1}$, so it stabilizes $\lambda_1$. 
Thus, from equation \eqref{eq:equation_pairswithhatwkminus1}, we get 
\[
\langle \lambda_{k+1},x\rangle \leq \langle \lambda_k,\Lambda\rangle
\]
for all $k=1,\ldots,n-1$.
 
Next theorem follows from Theorem \ref{thm:equations_DeltaK_G0Lambda} and the above calculation.

\begin{thm}
For $G = SU(n,1)$ and $\Lambda\in\Chol$, we have
\[
\Delta_K(G\cdot\Lambda) = \bigl\{\xi\in\got{t}^*_+\ | \,  \langle\lambda_1,\xi\rangle\geq\langle\lambda_1,\Lambda\rangle\geq\langle\lambda_2,\xi\rangle\geq\ldots\geq\langle\lambda_n,\xi\rangle\geq\langle\lambda_n,\Lambda\rangle\bigr\}.
\]
\end{thm}

\begin{figure}[htb]
\begin{center}
\setlength{\unitlength}{1cm}
\resizebox{8cm}{8cm}{
\begin{pspicture}(10,10)
\definecolor{Couleur1}{rgb}{0,0.5,0.7}
\definecolor{Couleur2}{rgb}{0.6,0.4,0.2}

\pspolygon[linecolor=lightgray,fillstyle=solid,linewidth=0pt,fillcolor=lightgray](2,0)(2,10)(3.5,10)(4.5,9.9)(5,9.8)(6,9.6)(7,9.2)(8,8.6)(8.8,7.8)(9.3,7)(9.7,6.3)(10,5)(9.7,3.7)(9.3,3)(8.8,2.2)(8,1.4)(7,0.8)(6,0.4)(5,0.2)(4.5,0.1)(3.5,0)
\pspolygon[linecolor=gray,fillstyle=solid,linewidth=0pt,fillcolor=gray](2,1)(2,9.8)(3.5,9.8)(4.5,9.7)(5.5,9.4)(6.5,9)(7.6,8.4)(8.5,7.5)(9.1,6.4)(9.4,5.272)
\pspolygon[linecolor=lightgray,fillstyle=solid,linewidth=0pt,fillcolor=lightgray](4,3.5)(2,6.964)(3.6,9.735)(4.4,9.65)(5,9.5)(5.8,9.25)(6.3,9.05)(6.7,8.85)(7,8.696)


\psdot(4,3.5)

\psline[linewidth=1pt](2,0)(2,10)
\psline[linewidth=1pt](0.268,0)(9.9,5.561)
\psline[linewidth=1pt](4,3.5)(2,6.964)
\psline[linewidth=1pt](4,3.5)(7,8.696)
\psline[linewidth=1pt](2,6.964)(3.6,9.735)

\psline[linewidth=1pt]{->}(2,1)(3,1)
\psline[linewidth=1pt]{->}(2,1)(2.5,1.87)
\psline[linewidth=1pt]{->}(2,1)(1.5,1.87)

\rput(4,3.2){$\Lambda$}
\rput(2.9,0.7){$\alpha$}
\rput(2.8,2){$\beta_{1}$}
\rput(1.2,2){$\beta_{2}$}
\rput(7,2){$\got{t}^*_+$}
\rput(8,5.5){$\Chol$}
\rput(4,7){$\Delta_K(G\cdot\Lambda)$}

\end{pspicture}
}
\end{center}
\caption{Polyhedron $\Delta_K(SU(2,1)\cdot\Lambda)$ for $\Lambda\in\Chol$}
\end{figure}

\begin{rem}
Almost all the computations above are the same when we consider the moment polyhedron $\Delta_{U(n)}(U(n)\cdot\Lambda\times(\C^n)^*)$, where $U(n)$ acts canonically on $(\C^n)^*$, and $\Lambda$ is an element of $\got{t}_+^*$. The only difference appears in the set of dominant indivisible $(\C^n)^*$-admissible one parameter subgroups, being in this case the set $\{(1,0,\ldots,0),(0,\ldots,-1)\}$. Thus, Corollary \ref{cor:generalequations_DeltaKxK(T*KxE)} yields
\[
\Delta_{U(n)}(U(n)\cdot\Lambda\times(\C^n)^*) = \bigl\{\xi\in\got{t}^*_+\ | \  \xi_1\geq\Lambda_1\geq\xi_2\geq\ldots\geq\xi_n\geq\Lambda_n\bigl\}.
\]
This gives another proof that the irreducible representation $V_{\mu}$ with highest weight $\mu=(\mu_1\geq\ldots\geq\mu_n)$ of $GL_n(\C)$ appears in the decomposition into irreducible representation of $V_{\Lambda}\otimes\mathrm{Sym}(\C^n)$ if and only if
\[
\mu_1\geq\Lambda_1\geq\mu_2\geq\ldots\geq\mu_n\geq\Lambda_n,
\]
from \cite{macdonald,brion}. See \cite[9.3]{woodward} for more details.
\end{rem}

\subsubsection{Moment polyhedron of $SU(2,2)$}

Now, take $G=SU(2,2)$. Let $K$ be the connected Lie subgroup of $G$ with Lie algebra
\[
\got{k} := \left\{\left(\begin{array}{cc}A & 0 \\ 0 & B\end{array}\right)\ |\ A,B\in\got{su}(2), \tr(A)+\tr(B)=0\right\},
\]
and $T$ be the maximal torus of diagonal matrices in $SU(2,2)$. We define, for all $i=1,\ldots,4$, the linear forms $e_i^*$ on $\got{t}$, the Lie algebra of $T$, by taking $e_i^*(H) = h_i$ for any element $H=i\diag(h_1,h_2,h_3,h_4)$ in $\got{t}$. Then, the set of roots of $\got{g}$ is $\got{R}=\{\pm\alpha_{i,j} = e_i^*-e_j^*\ | \ 1\leq i,j\leq 4\}$. Moreover, $\got{R}_c^+ = \{\alpha_{1,2},\alpha_{3,4}\}$ and $\got{R}_n^+=\{\alpha_{i,j}\ | \ 1\leq i\leq 2, 3\leq j\leq 4\}$.

The dominant indivisible $\got{p}^-$-admissible one parameter subgroups are the elements of the set
\begin{multline*}
\{\lambda_1=(1,-3,1,1), \lambda_2=(-1,-1,3,-1), \lambda_3=(1,-1,1,-1),\\
 \lambda_4=(1,1,1,-3), \lambda_5=(3,-1,-1,-1)\}.
\end{multline*}
The one parameter subgroups $\lambda_1$ and $\lambda_2$ yield $\got{p}^-_{<0}=0$, hence any corresponding pair is well covering. For $\lambda_4$, we can easily check that $\dim_{\C}(\got{p}^-_{<0}) = 2$, $\langle\lambda_4,-\alpha_{1,4}\rangle = \langle\lambda_4,-\alpha_{2,4}\rangle = -4$, with $\Theta(\alpha_{1,4}\alpha_{2,4})=0$, and similarly for $\lambda_5$. Thus, it remains to compute the well covering pairs corresponding to $\lambda_3$.

For $\lambda_3$, we have $\dim_{\C}(\got{p}^-_{<0}) = 1$, $\langle\lambda,-\alpha_{1,4}\rangle = -2$, and $\Theta(\alpha_{1,4}) = \sigma^B_{(s_1,\id)}+\sigma^B_{(\id,s_1)}$. The cup product here is known, we have $\sigma^B_{(s_1,\id)}\,.\,\sigma^B_{(\id,s_1)} = \sigma^B_{(s_1,s_1)}$, and $\sigma^B_{(s_1,\,\cdot\,)}\,.\,\sigma^B_{(s_1,\,\cdot\,)} = 0 = \sigma^B_{(\,\cdot\,,s_1)}\,.\,\sigma^B_{(\,\cdot\,,s_1)}$. By a straightforward computation, we get the following set,
\[
\left\{\bigl((s_1,\id),(s_1,s_1)\bigr),\bigl((\id,s_1),(s_1,s_1)\bigr),\bigl((s_1,s_1),(s_1,\id)\bigr),\bigl((s_1,s_1),(\id,s_1)\bigr)\right\},
\]
which parametrizes the well covering pairs corresponding to $\lambda_3$. From Theorem \ref{thm:equations_DeltaK_G0Lambda}, and the fact that $\sum\xi_i=0$ if $\xi=(\xi_1,\ldots,\xi_4)$ is in $\got{t}^*$, we deduce the following statement.

\begin{thm}
\label{thm:equations_DeltaK_SU(2,2)}
For $G = SU(2,2)$ and $\Lambda\in\Chol$, the polyhedron $\Delta_K(G\cdot\Lambda)$ is defined by the following equations
\[
\left\{\begin{aligned}
\xi_1 \geq \Lambda_1,\ \xi_2\geq \Lambda_2,&\quad \xi_3 \leq \Lambda_3,\ \xi_4 \leq \Lambda_4 \\
|\xi_1 - \xi_2 - \xi_3 + \xi_4| & \leq \Lambda_1 - \Lambda_2 + \Lambda_3 - \Lambda_4 \\
-\xi_1 + \xi_2 - \xi_3 + \xi_4 & \leq -|\Lambda_1 - \Lambda_2 - \Lambda_3 + \Lambda_4| \\
\xi_1\geq\xi_2& ,\quad \xi_3\geq \xi_4
\end{aligned}\right.
\]
\end{thm}

In this setting, the holomorphic chamber is
\[
\Chol = \{\xi=(\xi_1,\ldots,\xi_4)\in\got{t}^*\ | \ \xi_1\geq\xi_2\geq\xi_3\geq\xi_4\}.
\]
If $\Lambda\in\Chol$, we easily check that $\Delta_K(G\cdot\Lambda)\subset\Chol$, using the first line of equations of Theorem \ref{thm:equations_DeltaK_SU(2,2)}.


\appendix

\section{Combinatorics in the Weyl group of $GL_r(\C)$}

\small

\subsection{Some properties of the Weyl group of $\hKC = GL_r(\C)$}
\label{subsection:properties_Weylgroup_GLr}

This section collects several properties about certain elements of the Weyl group $\wh{W}$ of $\hKC = GL_r(\C)$. These results are very useful in the proof of Theorem \ref{thm:nsc_wellcoveringpair}.

Let $\hat{\got{t}}_{\C}$ be the set of diagonal matrices of $\hat{\got{k}}_{\C}$. The roots of $\hat{\got{k}}_{\C}$ are the linear forms $\hat{\alpha}_{i,j}(\diag(h_1,\ldots,h_n)) = h_i - h_j$ on $\hat{\got{t}}_{\C}$,
where $1\leq i,j\leq r$, with $i\neq j$. The simple roots are the roots $\halpha_{i,i+1}$, where $i\in\{1,\ldots,r-1\}$. We denote by $s_{\halpha_{i,j}}$ the element of $\wh{W}$ associated to the root $\halpha_{i,j}$, and $s_i$ for the simple root $\halpha_{i,i+1}$.

We define $\check{w}_k = s_{r-1}s_{r-2}\ldots s_{k+1}s_k$ for any $k\in\{1,\ldots,r-1\}$, and $\check{w}_{r} = \id$. Let us recall that the elements $\hat{w}_k$ have been defined in the previous section by $\hat{w}_1 = \id$, and $\hat{w}_k = s_1\ldots s_{k-1}$ if $2\leq k\leq r$. An easy computation gives the lengths of these elements of $\wh{W}$.

\begin{lem}
\label{lem:wcheck_and_what_length}
For all $k=1,\ldots,r$, we have $l(\check{w}_k) = r-k$, and $l(\hat{w}_k) = k-1$.
\end{lem}

Let $\wh{Q}$ be the maximal parabolic subgroup of $\hKC$ defined by \eqref{eq:parabolicsubgroup_wh{Q}}.

\begin{lem}
\label{lem:shorter_elements_of_W_mod_WQ_right}
The set $\wh{W}_{\wh{Q}}\backslash\wh{W}$ has exactly $r$ classes, and the elements $\hat{w}_k$, for all $k=1,\ldots,r$, form a set of shortest representatives of each class. More precisely, for all $w\in\wh{W}_{\wh{Q}}$, we have $l(ws_1\ldots s_k) = l(w) + k$.
\end{lem}

\begin{proof}
The first assertion is obvious. It is clear that $\hat{w}_1 = \id$ is the shortest element of its class. Fix $w\in\wh{W}_{\wh{Q}}$. From \cite[Proposition 2.72]{knapp}, $\wh{W}_{\wh{Q}}$ is generated by the elements $s_2,\ldots,s_{r-1}$. Hence, $w(\halpha_{1,k+1})$ is always positive. Thus, the root $ws_1\ldots s_{k-1}(\halpha_{k,k+1}) = w(\halpha_{1,k+1})$ is positive. As well, $w(\halpha_{1,2}) > 0$. By Lemma 2.71 in \cite{knapp}, we necessarily have $l((ws_1\ldots s_{k-1})s_k) = l(ws_1\ldots s_{k-1})+1$ for $k=2,\ldots,r-1$, and $l(ws_1) = l(w) +1$. An obvious induction yields the expected result, i.e. $l(ws_1\ldots s_k) = l(w)+k$ for all $k\in\{1,\ldots,r-1\}$. And, $\hat{w}_k$ is clearly the shortest element of its class in $\wh{W}_{\wh{Q}}\backslash\wh{W}$.
\end{proof}

Now, by Lemma \ref{lem:shorter_elements_of_W_mod_WQ_right}, it is clear that $\hat{w}_{\wh{Q}}\hat{w}_k$ is the longest element of $\wh{W}_{\wh{Q}}\hat{w}_k$, and, for $k=r$,
\[
l(\hat{w}_{\wh{Q}}\hat{w}_{r}) = l(\hat{w}_{\wh{Q}}s_1\ldots s_{r-1}) = l(\hat{w}_{\wh{Q}}) + r-1 = l(\hat{w}_0).
\]
Thus, $\hat{w}_0\hat{w}_{\wh{Q}} = s_{r-1}\ldots s_1$. The next lemma is obvious.

\begin{lem}
\label{lem:relation_between_wcheck_what}
For all $k\in\{1,\ldots,r\}$, we have $\hat{w}_0\hat{w}_{\wh{Q}}\hat{w}_k = \check{w}_k$.
\end{lem}


The next lemma results from Lemma \ref{lem:shorter_elements_of_W_mod_WQ_right}. Below, the element $\hat{w}_k^{-1}$ equals $\hat{w}_k^{-1}=s_{k-1}\ldots s_1$ for any $k=2,\ldots,r$, and $\hat{w}_1^{-1} = \id$.

\begin{lem}
\label{lem:shorter_elements_of_W/WQ}
The set $\wh{W}/\wh{W}_{\wh{Q}}$ has exactly $r$ classes, and the elements $\hat{w}_k^{-1}$, for all $k=1,\ldots,r$, form a set of shortest representatives of each class. More precisely, we have $l(s_k\ldots s_1w) = l(w) + k$, for all $w\in\wh{W}_{\wh{Q}}$. Moreover, for all $k=1,\ldots,r$, we have $w_0\hat{w}_k^{-1}w_{\wh{Q}} = \hat{w}_{r-k+1}^{-1}$.
\end{lem}




The proofs of the next three lemmas are simple verifications.

\begin{lem}
\label{lem:wcheck_plus2}
Fix $1\leq i\leq k<j\leq r$. Then $l(\check{w}_ks_{\halpha_{i,j}}) = l(\check{w}_k)+1$ if and only if $(i,j)=(k-1,k+1)$. Moreover, $\check{w}_ks_{\halpha_{k-1,k+1}} = \check{w}_{k+1}s_{k-1}s_k = s_{r-1}\ldots s_{k+1}s_{k-1}s_k$.
\end{lem}


%

%

\begin{lem}
\label{lem:wcheck_plus2_jequalsk}
Let $k\in\{2,\ldots,r-1\}$. If $i<k-1$, then $l(\check{w}_k s_{\halpha_{i,k}}) \geq l(\check{w}_k) + 2$. For $i=k-1$, we have $\check{w}_k s_{\halpha_{k-1,k}} = \check{w}_{k-1}$, and $l(\check{w}_k s_{\halpha_{k-1,k}}) = l(\check{w}_k) + 1$
\end{lem}



\begin{lem}
\label{lem:wcheck-1_plus1_jequalsk+2}
Let $k\in\{1,\ldots,r-2\}$. Then, for all $j\in\{3,\ldots,r\}$, $l(s_k\ldots s_1 s_{\halpha_{1,j}}) = l(s_k\ldots s_1) + 1$ if and only if $j=k+2$, and $s_k\ldots s_1 s_{\halpha_{1,k+2}} = s_{k+1}s_k\ldots s_1$.
\end{lem}


\subsection{The Chevalley formula}
\label{subsection:Chevalleyformula}

\subsubsection{Statement of the Chevalley formula}
\label{subsubsection:statement_of_Chevalley_formula}

We keep the notation introduced in subsection \ref{subsection:main_criterion}. As noticed in \cite[A.2]{BS00} and in \cite{chevalley}, we may assume that $K$ is semisimple and simply connected.

The Chevalley formula is stated in the next theorem. We use the formulation given in \cite{BS00} (Theorem A.2.1). See also \cite{chevalley} and \cite{demazure} for a proof of the formula. We use the following notations. The vector $\alpha^{\vee}$ of $\got{t}$ is the coroot of $\alpha$ in $\got{t}$. Moreover, if $\alpha$ is a simple root, we denote by $\pi_{\alpha}$ the fundamental weight associated to $\alpha$. We define the morphism $\Theta : \wedge^*\rightarrow\mathrm{H}^2(K/T,\Z)$ which sends a weight $\mu$ of the weight lattice $\wedge^*$ of $T$, onto the first Chern class $\Theta(\mu) = c_1(\mathcal{L}_{\mu})$, of the line bundle $\mathcal{L}_{\mu}$ with weight $\mu$.

\begin{thm}[Chevalley]
\label{thm:Chevalley}
\begin{enumerate}
\item $\Theta$ is an isomorphism,
\item $\Theta(\pi_{\alpha}) = \sigma_{s_{\alpha}}$ for all simple root $\alpha$.
\item For all weight $\mu$ of $\wedge^*$,
\begin{equation}
\label{eq:Chevalley_formula}
\Theta(\mu).\sigma^{B}_{w} = \sum_{\stackrel{\alpha\in \got{R}^+}{l(w s_{\alpha}) = l(w)+1}}\mu(\alpha^{\vee})\sigma^{B}_{ws_{\alpha}}.
\end{equation}
\end{enumerate}
\end{thm}

Let $\hat{K}$ be another connected, simply connected, compact, semisimple, real Lie group, and $f:K\rightarrow\hat{K}$ be a Lie group homomorphism with finite kernel. The map $f$ induces a homomorphism $\KC\rightarrow\hKC$ with finite kernel, and an embedding $f^B:\KC/B\rightarrow \hKC/\hat{B}$, which induces a map $(f^B)^*:\coh{\hKC/\hat{B}}\rightarrow\coh{\KC/B}$ in cohomology. As said in \cite[A.2]{BS00}, the map $(f^B)^*$ is completely determined by the maps $f^*$, $\Theta$ and $\wh{\Theta}$.

\subsubsection{Computation of $(f_{\lambda}^B)^*(\check{w}_k)$}

%

\begin{lem}
\label{lem:cupproduct_sk_wcheckk}
For all $k=2,\ldots,r-1$, we have
\[
\sigma_{s_k}^{\wh{B}}.\sigma_{\check{w}_k}^{\wh{B}} = \sigma^{\wh{B}}_{s_{r-1}\ldots s_{k+2}s_{k+1}s_{k-1}s_k}.
\]
\end{lem}

\begin{proof}
A straightforward verification yields
\[
\sigma^{\wh{B}}_{s_k}.\sigma^{\wh{B}}_{\check{w}_k} = \sum_{\stackrel{i\leq k< j}{l(\check{w}_k s_{\halpha_{i,j}}) = l(\check{w}_k)+1}}\sigma^{\wh{B}}_{\check{w}_k s_{\halpha_{i,j}}},
\]
from \eqref{eq:Chevalley_formula}. Now the result directly follows from Lemma \ref{lem:wcheck_plus2}.
\end{proof}


\begin{lem}
\label{lem:cupproduct_skminus1_wcheckk}
For all $k=2,\ldots,r-1$, we have
\[
\sigma_{s_{k-1}}^{\wh{B}}.\sigma_{\check{w}_k}^{\wh{B}} = \sigma^{\wh{B}}_{s_{r-1}\ldots s_{k+2}s_{k+1}s_{k-1}s_k} + \sigma_{\check{w}_{k-1}}^{\wh{B}}.
\]
\end{lem}

\begin{proof}
This is similar to the proof of Lemma \ref{lem:cupproduct_sk_wcheckk}. Here we use Lemmas \ref{lem:wcheck_plus2} and \ref{lem:wcheck_plus2_jequalsk}.
\end{proof}

We now have the obvious following Theorem.

\begin{thm}
\label{thm:formume_sigmawcheckk}
For all $k=2,\ldots,r-1$, we have $\sigma_{\check{w}_{k-1}}^{\wh{B}} = (\sigma_{s_{k-1}}^{\wh{B}}-\sigma_{s_k}^{\wh{B}}).\sigma_{\check{w}_k}^{\wh{B}}$. Furthermore, we have
\begin{equation}
\label{eq:sigmawcheck_productofclassesofdegree2}
\sigma_{\check{w}_{k-1}}^{\wh{B}} = (\sigma_{s_{k-1}}^{\wh{B}}-\sigma_{s_k}^{\wh{B}}).(\sigma_{s_k}^{\wh{B}}-\sigma_{s_{k+1}}^{\wh{B}}).\cdots.(\sigma_{s_{r-2}}^{\wh{B}}-\sigma_{s_{r-1}}^{\wh{B}}).\sigma_{s_{r-1}}^{\wh{B}}.
\end{equation}
\end{thm}
%

Now, we can compute the images of the classes $\sigma_{\check{w}_{k-1}}^{\wh{B}}\in\coh{\hKC/\hat{B}}$ by the map $(f_{\lambda}^B)^*$, for all $k\in\{2,\ldots,r\}$. We consider the map
\[
f_{\lambda}:H\in\got{t}_{\C}\cap[\got{k}_{\C},\got{k}_{\C}]\mapsto\diag(\beta_1(H),\ldots,\beta_r(H))\in\hat{\got{t}}_{\C}\cap\got{sl}_r,
\]
where $\beta_1,\ldots,\beta_r$ are the weights of $\TC$ on the $\KC$-module $M$, following the parametrization of $M$ made in paragraph \ref{subsection:parametrization}. Let us denote by $\got{t}_{ss} = \got{t}\cap[\got{k},\got{k}]$ the semisimple part of $\got{k}$ in $\got{t}$. Thus, $f_{\lambda}^*(\hat{e}_i) = \beta_i|_{\got{t}_{ss}}$, and consequently, $f_{\lambda}^*(\pi_{\halpha_{i,i+1}}) = \sum_{j=1}^i \beta_j|_{\got{t}_{ss}}$, and also $f_{\lambda}^*(\pi_{\halpha_{i-1,i}}) - f_{\lambda}^*(\pi_{\halpha_{i,i+1}}) = -\beta_i|_{\got{t}_{ss}}$. Applying this to the equality \eqref{eq:sigmawcheck_productofclassesofdegree2} for any $k\leq r-1$, we get the following formula
\begin{equation}
\label{eq:sigmawcheck_image_by_flambdacheck}
(f_{\lambda}^B)^*(\sigma_{\check{w}_{k-1}}^{\wh{B}}) = \Theta\left((-\beta_k|_{\got{t}_{ss}}) \ldots (-\beta_{r-1}|_{\got{t}_{ss}}) (\sum_{i=1}^{r-1}\beta_i|_{\got{t}_{ss}})\right).
\end{equation}

\begin{lem}
Let $\zeta=\sum_{\beta\in\WT(E)}\beta$ be the sum of the weights of the action of $T$ on $E$. Then $\zeta|_{\got{t}_{ss}} = 0$.
\end{lem}

\begin{proof}
The representation $E$ of $\KC$ induces the representation $\det E = \Lambda^r E$ of $\KC$. This yields a character $\chi:\KC\rightarrow\C^*$, with derivative $d\chi_e:\got{k}_{\C}\rightarrow\C$, morphism of Lie algebras. But, here, $\C$ is a Abelian Lie algebra, thus, for all $X,Y\in\got{k}_{\C}$, we have $d\chi_e([X,Y]) = 0$. And it is obvious that we have $\zeta = \sum_{\beta\in\WT(E)}\beta = id\chi_e|_{\got{t}^*}$. Since $\got{t}_{ss} = \got{t}\cap[\got{k},\got{k}]$ , we conclude that $\zeta|_{\got{t}_{ss}} = 0$.
\end{proof}

In our case, $\zeta = \sum_{i=1}^r\beta_i$. Thus, $-\beta_r|_{\got{t}_{ss}} = \sum_{i=1}^{r-1}\beta_i|_{\got{t}_{ss}}$. We may replace this value in the equation 
 \eqref{eq:sigmawcheck_image_by_flambdacheck}
, which yields, for any integer $k\in\{2,\ldots,r\}$,
\begin{equation}
\label{eq:sigmawcheck_image_by_flambdacheck_final}
(f_{\lambda}^B)^*\left(\sigma_{\check{w}_{k-1}}^{\wh{B}}\right) = \Theta\left((-\beta_k|_{\got{t}_{ss}}) \ldots (-\beta_{r}|_{\got{t}_{ss}})\right).
\end{equation}


\normalsize

\bibliography{biblio}
\bibliographystyle{plain}

\end{document}